\documentclass{article}%
\usepackage{amsmath}
\usepackage{amsfonts}
\usepackage{amssymb}
\usepackage{graphicx}
\usepackage{color}
\usepackage{algorithmic}%
\providecommand{\U}[1]{\protect\rule{.1in}{.1in}}
\newtheorem{theorem}{Theorem}
\newtheorem{algorithm}[theorem]{Algorithm}

\newtheorem{corollary}[theorem]{Corollary}

\newtheorem{example}[theorem]{Example}

\newtheorem{lemma}[theorem]{Lemma}
\newtheorem{notation}[theorem]{Notation}

\newtheorem{remark}[theorem]{Remark}

\newenvironment{proof}[1][Proof]{\noindent\textbf{#1.} }{\ \rule{0.5em}{0.5em}}
\begin{document}

\title{Convergence analysis of energy conserving explicit local time-stepping methods
for the wave equation\thanks{M. Mehlin gratefully acknowledges financial
support by the Deutsche Forschungsgemeinschaft (DFG) through CRC 1173.}}
\author{Marcus J. Grote\thanks{Department of Mathematics and Computer Science,
University of Basel, Spiegelgasse 1, 4051 Basel, Switzerland,
(marcus.grote@unibas.ch)}
\and Michaela Mehlin\thanks{Institute for Applied and Numerical Analysis, Karlsruhe
Institute of Technology, 76131 Karlsruhe, Germany,(michaela.mehlin@kit.edu).}
\and Stefan A. Sauter\thanks{Institute for Mathematics, University of Zurich,
Winterthurerstrasse 190, 8057 Zurich, Switzerland, (stas@math.uzh.ch)}
}
\maketitle

\begin{abstract}
Local adaptivity and mesh refinement are key to the efficient simulation of wave phenomena
in heterogeneous media or complex geometry. Locally refined meshes, however, dictate
a small time-step everywhere with a crippling
effect on any explicit time-marching method. In \cite{DiazGrote09} a leap-frog (LF) based  explicit local time-stepping (LTS) method was proposed, which overcomes
the severe bottleneck due to a few small elements
by taking small time-steps in the locally refined region
and larger steps elsewhere. Here a rigorous convergence proof is presented
for the fully-discrete LTS-LF method
when combined with a standard conforming finite element method (FEM) in space. Numerical results
further illustrate the usefulness of the LTS-LF Galerkin FEM in the presence of corner singularities.

\end{abstract}


\vspace{0.5cm}

\noindent\textbf{AMS-Classification}:
65M12, 65M20, 65M60, 65L06, 65L20


\section{Introduction}
Efficient numerical methods are crucial for the
simulation of time-dependent acoustic, electromagnetic or elastic wave phenomena. Finite element
methods (FEM), in particular, easily accommodate varying mesh sizes or polynomial degrees. Hence, they are remarkably effective and widely used for the spatial discretization in heterogeneous media or
complex geometry. However, as spatial discretizations become increasingly accurate and flexible,
the need for more sophisticated time-integration methods for the resulting systems
of ordinary differential equations (ODE) becomes all the more
apparent.

Today's standard use of local adaptivity and mesh refinement
causes a severe bottleneck for any standard explicit time integration. Even if the refined
region consists of only a few small elements, those smallest elements will impose
a tiny time-step everywhere for stability reasons.
To overcome that geometry induced stiffness, various local
time integration strategies were devised in recent years.
Typically the mesh is partitioned into a ``coarse'' part, where most of the elements are located, and a ``fine'' part, which contains the remaining few smallest elements. Inside the ``coarse'' part, standard explicit methods are used for time integration. Inside the ``fine'' part, local time-stepping
(LTS) methods either use implicit or explicit time integration.

Locally implicit methods are based on implicit-explicit (IMEX) approaches commonly used in CFD
for operator splitting \cite{ARW95,HV03}.
They require the solution of a linear system inside the refined region at every
time-step, which becomes increasingly expensive (and ill-conditioned) as the mesh size decreases  \cite{KCGH07}. Alternatively, exponential Adams methods \cite{HO11} apply the matrix exponential locally in the fine part while reducing to the underlying Adams-Bashforth scheme elsewhere.

Locally implicit or exponential time integrators typically use the same time-step everywhere but
apply different methods in the "fine" and the "coarse" part. In contrast, explicit LTS methods typically
use the same method everywhere but take smaller time-steps inside the ``fine'' region  \cite{GearWells84}; hence,
they remain fully explicit. Since the finite-difference based adaptive mesh refinement (AMR) method by Berger and Oliger \cite{BergerOliger}, various explicit LTS were proposed in the context of discontinuous
Galerkin (DG) FEM, which permit a different time-step inside each individual element \cite{Flaherty97,LGM07,DKT07,TDMS09,CS07,CS09}. In \cite{DNBH15} multiple time-stepping algorithms were presented which allow any choice of explicit Adams type or predictor-corrector scheme for the integration of the coarse region and any choice of ODE solver for the integration of the fine part. High-order explicit LTS methods for wave propagation were derived in \cite{Grote_Mitkova,GM13,GMM15} starting either from Leap-Frog,
Adams-Bashforth or Runge-Kutta methods.

In~\cite{ColFouJol3,BecJolyRodri,ColFouJol1}, Collino et al.
proposed a first energy conserving LTS method for the wave equation which
was analyzed
in~\cite{ColFouJol2, JolyRodriguez}. This second-order method conserves a discrete energy
and thereby guarantees stability, but it requires at every time-step the solution of a linear
system at the interface between the fine and the coarser elements; hence, it is not fully explicit.
A fully explicit second-order LTS method was proposed for Maxwell's equations by Piperno ~\cite{Piperno}
and further developed in~\cite{DFFL10, MPFC08}. In  \cite{MZKM13,RPSUG15},
the high-order energy conserving explicit LTS
method proposed in \cite{DiazGrote09} was successfully applied to 3D seismic wave propagation
on a large-scale parallel computer architecture.

Despite the many different explicit LTS methods that were proposed and successfully used for
wave propagation in recent years,
 a rigorous fully discrete space-time convergence theory is still lacking.
In fact, convergence has been proved only for the method of Collino et al. \cite{ColFouJol2,ColFouJol3,JolyRodriguez} and very recently for the locally implicit
method for Maxwell's equations by Verwer \cite{Ver09,DLM13,HS16},
neither fully explicit. Indeed, the difficulty in proving
convergence of fully explicit LTS methods is twofold. On the one hand, classical
proofs of convergence \cite{Dupont,Baker} always assume standard time discretizations,
while proofs for multirate schemes (in the ODE literature)
are always restricted to the finite-dimensional case.
Hence, standard convergence analysis cannot be easily extended
to LTS methods for partial differential equations. On the other hand, when explicit LTS schemes are
reformulated as perturbed one-step schemes,
they involve products of differential and restriction operators, which
do not commute and seem to inevitably lead to a loss of regularity.

Our paper is structured as follows. In Section \ref{SecGalDisc}, we consider a
general second-order wave equation and introduce (the notation for) conforming finite
element spaces on simplicial meshes with local polynomial order $m$. Next, we define
finite-dimensional restriction operators to the "fine" grid and formulate the leap-frog (LF) based LTS method from \cite{DiazGrote09} in a Galerkin conforming finite element setting. In Section \ref{SecStabCons}, we prove continuity and
coercivity estimates for the LTS operator that are robust with respect to the
number of local time-steps $p$, provided a genuine CFL condition is satisfied. Here,
new estimates on the coefficients that appear when rewriting the LTS-LF scheme in "leap-frog
manner"  play a key-role -- see Appendix. Those estimates
pave the way for the stability estimate of the time iteration operator, for
which we then prove a stability bound independently of $p$.
In doing so, the truncation errors are estimated through standard Taylor arguments
for the leap-frog method. Due to the local restriction, however,
a judicious splitting of the iteration operator and its inverse is required
to avoid negative powers of $h$ via inverse inequalities. By combining
our analysis of the semi-discrete formulation, which takes into account the effect of
local time-stepping, with classical error estimates \cite{Baker}, we eventually
obtain optimal convergence rates explicit with respect to the time step
$\Delta t$, the mesh size $h$, the right-hand side, the initial data and the final time $T$,
which hold uniformly with respect to the number of local time-steps $p$.
Finally, in Section \ref{SecNumExp}, we report on some numerical experiments inside an L-shaped domain. By applying the LTS method in the locally refined region near the re-entrant corner,
we obtain a significant speedup over a standard leap-frog
method with a small time-step everywhere.


\section{Galerkin Discretization with Leap-Frog Based Local Time-Stepping}\label{SecGalDisc}

\subsection{The Wave Equation}

Let $\Omega\subset\mathbb{R}^{d}$ be a Lipschitz domain
and $L^{2}\left(  \Omega\right)  $ denote the space of
square integrable, real-valued functions with scalar product denoted by
$\left(  \cdot,\cdot\right)  $ and corresponding norm by $\left\Vert
\cdot\right\Vert =\left(  \cdot,\cdot\right)  ^{1/2}$. Next, let $H^{1}\left(
\Omega\right)  $ denote the standard Sobolev space of all square integrable,
real-valued functions whose first (weak) derivatives are also square
integrable; as usual, $H^{1}\left(  \Omega\right)  $ is equipped with the norm
$\left\Vert u\right\Vert _{H^{1}\left(  \Omega\right)  }= (  \left\Vert
u\right\Vert ^{2}+\left\Vert \nabla u\right\Vert ^{2})^{1/2}$.

We now let
$V\subset H^{1}\left(  \Omega\right)  $ denote a closed subspace of
$H^{1}\left(  \Omega\right)  $, such as $V=H^{1}\left(
\Omega\right)  $ or $V=H_{0}^{1}\left(  \Omega\right)  $, and consider a
bilinear form $a:V\times V\rightarrow\mathbb{R}$ which is symmetric,
continuous, and coercive:%
\begin{subequations}
\label{wellposed}
\end{subequations}%
\begin{equation}
a\left(  u,v\right)  =a\left(  v,u\right)  \qquad\forall u,v\in V \tag{%
\ref{wellposed}%
a}\label{wellposeda}%
\end{equation}
and%
\begin{equation}
\left\vert a\left(  u,v\right)  \right\vert \leq C_{\operatorname*{cont}%
}\left\Vert u\right\Vert _{H^{1}\left(  \Omega\right)  }\left\Vert
v\right\Vert _{H^{1}\left(  \Omega\right)  }\qquad\forall u,v\in V \tag{%
\ref{wellposed}%
b}\label{wellposedb}%
\end{equation}
and%
\begin{equation}
a\left(  u,u\right)  \geq c_{\operatorname*{coer}}\left\Vert u\right\Vert
_{H^{1}\left(  \Omega\right)  }^{2}\qquad\forall u\in V. \tag{%
\ref{wellposed}%
c}\label{wellposedc}%
\end{equation}
For given $u_{0}\in V, v_{0}\in L^{2}\left(  \Omega\right)  $ and $F:\left[
0,T\right]  \rightarrow V^{\prime}$, we consider the wave equation: Find
$u:\left[  0,T\right]  \rightarrow V$ such that%
\begin{equation}
\left(  \ddot{u},w\right)  +a\left(  u,w\right)  =F\left(  w\right)
\quad\forall w\in V,t>0 \label{waveeq}%
\end{equation}
with initial conditions%
\begin{equation}
u\left(  0\right)  =u_{0}\quad\text{and\quad}\dot{u}\left(  0\right)  =v_{0}.
\label{waveeqic}%
\end{equation}
It is well known that (\ref{waveeq})--(\ref{waveeqic}) is well-posed for sufficiently regular $u_0$, $v_0$ and $F$
\cite{LionsMagenesI}. In fact, the weak solution $u$ can be shown to be
continuous in time, that is, $u\in C^{0}(0,T;V),\dot{u}\in C^{0}%
(0,T;L^{2}\left(  \Omega\right)  )$ -- see [\cite{LionsMagenesI}, Chapter III,
Theorems 8.1 and 8.2] for details -- which implies that the
initial conditions (\ref{waveeqic}) are well defined.

\begin{example}
\label{Exmodel problem}The classical second-order wave equation in strong form
is given by%
\begin{equation}%
\begin{split}
u_{tt}-\nabla\cdot(c^{2}\nabla u)  &  =f\qquad\;\,\mbox{in }\Omega
\times(0,T),\\
u  &  =0\qquad\;\,\mbox{on }\Gamma_{D}\times(0,T),\\
\frac{\partial u}{\partial\nu}  &  =0\qquad\;\,\mbox{on }\Gamma_{N}%
\times(0,T),\\
u|_{t=0}  &  =u_{0}\qquad\mbox{in }\Omega,\\
u_{t}|_{t=0}  &  =u_{0}\qquad\mbox{in }\Omega.\\
&
\end{split}
\label{model problem}%
\end{equation}
In this case, we have $V:=H_{D}^{1}\left(  \Omega\right)  :=\left\{  w\in
H^{1}\left(  \Omega\right)  :\left.  w\right\vert _{\Gamma_{D}}=0\right\}  $;
the bilinear form is given by $a\left(  u,v\right)  :=\left(  c^{2}\nabla
u,\nabla u\right)  $ and the right-hand side by $F\left(  w\right)  =\left(
f,w\right)  $ for all $w\in V$.
\end{example}

\subsection{Galerkin Finite Element Discretization}

For the semi-discretization in space, we employ the Galerkin finite element
method and we first have to introduce some notation. We assume for the spatial
dimension $d\in\left\{  1,2,3\right\}  $ and that the bounded Lipschitz domain
$\Omega\subset\mathbb{R}^{d}$ is an interval for $d=1$, a polygonal domain for
$d=2$, and a polyhedral domain for $d=3$. Let $\mathcal{T}:=\left\{  \tau
_{i}:1\leq i\leq N_{\mathcal{T}}\right\}  $ denote a conforming (i.e.: no
hanging nodes), simplicial finite element mesh for $\Omega$. Let%
\[
h_{\tau}:=\operatorname*{diam}\tau\quad\text{and\quad}h:=\max_{\tau
\in\mathcal{T}}h_{\tau}\quad\text{and\quad}h_{\min}:=\min_{\tau\in\mathcal{T}%
}h_{\tau}%
\]
and denote by $\rho_{\tau}$ the diameter of the largest inscribed ball in
$\tau$. As a convention, the simplices $\tau\in\mathcal{T}$ are closed sets.
The shape regularity constant $\gamma$ of the mesh $\mathcal{T}$ is defined by%
\[
\gamma\left(  \mathcal{T}\right)  :=\max_{\tau}\left\{
\begin{array}
[c]{ll}%
\max\left\{  \frac{h_{\tau}}{h_{t}}:t\in\mathcal{T}:t\cap\tau\neq
\emptyset\right\}  & d=1,\\
\dfrac{h_{\tau}}{\rho_{\tau}} & d=2,3,
\end{array}
\right.
\]
and the quasi-uniformity constant by%
\[
C_{\operatorname*{qu}}:=\frac{h}{h_{\min}}.
\]

For $m\in\mathbb{N}$, we define the continuous, piecewise polynomial finite
element space by%
\[
S_{\mathcal{T}}^{m}:=\left\{  u\in C^{0}\left(  \Omega\right)  \mid\forall
\tau\in\mathcal{T}:\left.  u\right\vert _{\tau}\in\mathbb{P}_{m}\right\}  ,
\]
where $\mathbb{P}_{m}$ is the space to $d$-variate polynomials of maximal
total degree $m$. The definition of a Lagrangian nodal basis is standard and
employs the concept of a reference element. Let%
\[
\hat{\tau}:=\left\{  \mathbf{x}=\left(  x_{i}\right)  _{i=1}^{d}\in
\mathbb{R}_{\geq0}^{d}:\sum_{i=1}^{d}x_{i}\leq1\right\}
\]
denote the reference element. For $\tau\in\mathcal{T}$, let $\phi_{\tau
}:\widehat{\tau}\rightarrow\tau$ denote an affine pullback. For $m\geq1$,
we denote by $\hat{\Sigma}^{m}$ a set of nodal points in $\hat{\tau}$ unisolvent on $\mathbb{P}_{m}$, which allow to impose
continuity across simplex faces. The nodal points on a simplex $\tau\in\mathcal{T}$ are then
given by lifting those of the reference element:%
\[
\Sigma_{\tau}^{m}:=\left\{  \phi_{\tau}\left(  z\right)  :z\in\hat{\Sigma}%
^{m}\right\}  .
\]
The set of global nodal points is given by%
\[
\Sigma_{\mathcal{T}}^{m}:=%
{\textstyle\bigcup\nolimits_{\tau\in\mathcal{T}}}
\Sigma_{\tau}^{m}.
\]
A Lagrange basis for $S_{\mathcal{T}}^{m}$ is given by $\left(  b_{z,m}%
\right)  _{z\in\Sigma_{\mathcal{T}}^{m}}$ via the conditions%
\[
b_{z,m}\in S_{\mathcal{T}}^{m}\quad\text{and\quad}\forall z^{\prime}\in
\Sigma_{\mathcal{T}}^{m}\text{ it holds }b_{z,m}\left(  z^{\prime}\right)
=\left\{
\begin{array}
[c]{ll}%
1 & z=z^{\prime},\\
0 & \text{otherwise.}%
\end{array}
\right.
\]

For a subset $\Sigma\subset\Sigma_{\mathcal{T}}^{m}$, we define a
\textit{prolongation map} $P_{\Sigma}:\mathbb{R}^{\Sigma}\rightarrow
S_{\mathcal{T}}^{m}$ and a \textit{restriction map} $\mathbf{R}_{\Sigma
}:S_{\mathcal{T}}^{m}\rightarrow\mathbb{R}^{\Sigma}$ by%
\[
P_{\Sigma}\mathbf{u}=\sum_{z\in\Sigma}u_{z}b_{z,m}\quad\text{and\quad}\left(
\mathbf{R}_{\Sigma}v\right)  =\left(  \int_{\Omega}vb_{z,m}\right)
_{z\in\Sigma}.
\]
The mass matrix, $\mathbf{M}_{\Sigma}$, is given by%
\[
\mathbf{M}_{\Sigma}:=\left(  \int_{\Omega}b_{z,m}b_{z^{\prime},m}\right)
_{z,z^{\prime}\in\Sigma}\text{.}%
\]
If $\Sigma=\Sigma_{\mathcal{T}}^{m}$ holds, we write $P,\mathbf{R}$,
$\mathbf{M}$ short for $P_{\Sigma},\mathbf{R}_{\Sigma}$, $\mathbf{M}_{\Sigma}$.

\begin{remark}
Since $\mathbf{M}_{\Sigma}=\mathbf{R}_{\Sigma}P_{\Sigma}$, we also have
$P_{\Sigma}^{-1}=\mathbf{M}_{\Sigma}^{-1}\mathbf{R}_{\Sigma}$.
\end{remark}

The matrix $\mathbf{M}_{\Sigma}$ is the matrix representation of the $L^{2}%
$-scalar product with respect to the basis $\left(  b_{z,m}\right)
_{z\in\Sigma}$. We introduce a diagonally weighted, mesh dependent Euclidean
scalar product which is equivalent to the bilinear form $\left\langle
\mathbf{u},\mathbf{M}_{\Sigma}\mathbf{v}\right\rangle $ (cf. Lemma
\ref{Satz9.1}), where $\left\langle \cdot,\cdot\right\rangle $ denotes the
Euclidean scalar product on $\mathbb{R}^{\Sigma}$.

For $u=P\mathbf{u}$ and $v=P\mathbf{v}$ with $\mathbf{u}=\left(  u_{z}\right)
_{z\in\Sigma_{\mathcal{T}}^{m}}$ and $\mathbf{v}=\left(  v_{z}\right)
_{z\in\Sigma_{\mathcal{T}}^{m}}$ we set%
\[
\left(  u,v\right)  _{\mathcal{T}}:=\sum_{\tau\in\mathcal{T}}\left\vert
\tau\right\vert \sum_{\mathbf{z}\in\Sigma_{\tau}^{m}}u_{z}v_{z}=\left\langle
\mathbf{D}_{\Sigma_{\mathcal{T}}^{m}}\mathbf{u},\mathbf{v}\right\rangle
\text{\quad with\quad}\left\{
\begin{array}
[c]{l}%
\mathbf{D}_{\Sigma_{\mathcal{T}}^{m}}=\operatorname*{diag}\left[  d_{z}%
:z\in\Sigma_{\mathcal{T}}^{m}\right]  ,\\
d_{z}:=\left\vert \operatorname*{supp}b_{z,m}\right\vert ,
\end{array}
\right.
\]
where, for a measurable set $\omega\subset\mathbb{R}^{d}$, we denote by
$\left\vert \omega\right\vert $ its $d$-dimensional volume. The norm is given
by%
\[
\left\Vert u\right\Vert _{\mathcal{T}}:=\left(  u,u\right)  _{\mathcal{T}%
}^{1/2}.
\]
For later use, we define a localized version of $\mathbf{D}_{\Sigma
_{\mathcal{T}}^{m}}$. Let $\mathcal{N}\subset\Sigma_{\mathcal{T}}^{m}$ and
define the diagonal matrix $\mathbf{D}_{\mathcal{N}}=\operatorname{diag}%
\left[  d_{\mathcal{N},z}:z\in\Sigma_{\mathcal{T}}^{m}\right]  $ by%
\[
d_{\mathcal{N},z}:=\left\{
\begin{array}
[c]{ll}%
d_{z} & z\in\mathcal{N},\\
0 & z\in\Sigma_{\mathcal{T}}^{m}\backslash\mathcal{N}.
\end{array}
\right.
\]
We define the \textit{fine grid restriction operator} $R_{\mathcal{N}%
}:S_{\mathcal{T}}^{m}\rightarrow S_{\mathcal{T}}^{m}$ by%
\begin{equation}
R_{\mathcal{N}}= \mathbf{R}^{-1}\mathbf{D}_{\mathcal{N}} P^{-1}.
\label{defrn1}%
\end{equation}

\begin{remark}
Note that  the diagonal matrix $\mathbf{D}_{\mathcal{N}}$ corresponds to the matrix representation of $R_{\mathcal{N}}$:
\begin{equation}
\left(  R_{\mathcal{N}}P\mathbf{u},P\mathbf{v}\right)  =\left\langle
\mathbf{D}_{\mathcal{N}}\mathbf{u},\mathbf{v}\right\rangle =\sum
_{z\in\mathcal{N}}d_{z}u_{z}v_{z}. \label{symmR}%
\end{equation}
For the support of $R_{\mathcal{N}}u$ it holds%
\[
\operatorname*{supp}\left(  R_{\mathcal{N}}u\right)  \subset\Omega
_{\mathcal{N}}:=%
{\textstyle\bigcup\nolimits_{\substack{\tau\in\mathcal{T}\\\tau\cap
\mathcal{N}\neq\emptyset}}}
\tau.
\]
The operator $R_{\mathcal{N}}$ is symmetric positive semi-definite, which follows from $d_z \geq 0$ and the symmetry of
the right-hand side in (\ref{symmR}).
\end{remark}

We define \textit{conforming subspaces} of $V$ by%
\[
V_{\mathcal{T}}^{m}:=S_{\mathcal{T}}^{m}\cap V\text{.}%
\]

\begin{notation}
We write $S$ short for $V_{\mathcal{T}}^{m}$ if no confusion is possible.
Since $S=S_{\mathcal{T}}^{m}\cap V,$ we may assume that there is a subset
$\Sigma_{S}\subset\Sigma_{\mathcal{T}}^{m}$ such that $S=\operatorname{span}%
\left\{ b_{z,m}:z\in\Sigma_{S}\right\}  $.

The operators associated to the continuous and discrete bilinear form are the
linear mappings $A:V\rightarrow V^{\prime}$ and $A_{S}:S\rightarrow S$ defined
by%
\begin{align*}
\left\langle Au,v\right\rangle _{V^{\prime}\times V}  &  =a\left(  u,v\right)
\qquad\forall u,v\in V,\\
\left(  A_{S}u,v\right)   &  =a\left(  u,v\right)  \qquad\forall u,v\in S.
\end{align*}
Here $\left\langle \cdot,\cdot\right\rangle _{V^{\prime}\times V}$ is the
continuous extension of the $L^{2}\left(  \Omega\right)  $ scalar product to
the dual pairing $\left\langle \cdot,\cdot\right\rangle _{V^{\prime}\times V}$.
\end{notation}

\begin{example}
If homogeneous Dirichlet boundary conditions are imposed for the wave equation
we have $V:=H_{0}^{1}\left(  \Omega\right)  :=\left\{  u\in H^{1}\left(
\Omega\right)  \mid\left.  u\right\vert _{\partial\Omega}=0\right\}  $. The
nodal points $\Sigma_{\mathcal{T}}^{1}$ for the $\mathbb{P}_{1}$ finite
element space are the inner triangle vertices and $b_{z,1}$ is the usual
continuous, piecewise affine basis function for the nodal point $z$.
\end{example}

The semi-discrete wave equation then is given by: find $u_{S}:\left[
0,T\right]  \rightarrow S$ such that%
\begin{subequations}
\label{spacedisc}
\end{subequations}%
\begin{equation}
\left(  \ddot{u}_{S},v\right)  +a\left(  u_{S},v\right)  =F\left(  v\right)
\quad\forall v\in S,t>0 \tag{%
\ref{spacedisc}%
a}\label{spacedisca}%
\end{equation}
with initial conditions%
\begin{equation}
\left.
\begin{array}
[c]{c}%
\left(  u_{S}\left(  0\right)  ,w\right)  =\left(  u_{0},w\right) \\
\\
\left(  \dot{u}_{S}\left(  0\right)  ,w\right)  =\left(  v_{0},w\right)
\end{array}
\right\}  \quad\forall w\in S. \tag{%
\ref{spacedisc}%
b}\label{spacediscb}%
\end{equation}

\subsection{Discrete LTS-Galerkin FE Formulation}

Starting from the leap-frog based local time-stepping LTS-LF scheme from  \cite{DiazGrote09},
we now present the fully discrete space-time Galerkin FE formulation. First we let the (global) time-step
$\Delta t=T/N$ and denote by $u_{S}^{\left(n\right)  } = P \mathbf{u}_{S}^{\left(  n\right)  }$ the FE approximation at time $t_{n}=n\Delta t$ for the corresponding coefficient vector (nodal values) $\mathbf{u}_{S}^{\left(  n\right)  }\in \mathbb{R}^{\Sigma}$ .
Similarly we define the right-hand sides $f_{S}:\left[  0,T\right]  \rightarrow S$ and $f_{S}^{\left(  n\right)
}\in S$ by
\begin{equation}
\left(  f_{S},w\right)  =F\left(  w\right)  \quad\forall w\in S\quad
\text{and\quad}f_{S}^{\left(  n\right)  }:=f_{S}\left(  t_{n}\right)
\text{,}\label{deffnsn}%
\end{equation}
where again
$f_{S}^{\left(  n\right)  }=P\,\mathbf{f}_{S}^{\left(  n\right)  }$ with corresponding coefficients  $\mathbf{f}_{S}^{\left(  n\right)  }\in\mathbb{R}^{\Sigma}$.

Given the numerical
solution at times $t_{n-1}$ and $t_{n}$, the LTS-LF method then computes the numerical
solution of \eqref{spacedisc} at $t_{n+1}$ by using a smaller
time-step $\Delta\tau=\Delta t/p$ inside the regions of local refinement; here, $p \geq 2$ denotes
the "coarse" to "fine" mesh size ratio.
Clearly, if the maximal velocity in the coarse and the fine regions
differ significantly, the choice of $p$ should also reflect that variation
and instead denote the local CFL number ratio.
In the "fine" region, the right-hand side is also
evaluated at the intermediate times $t_{n+\frac{m}{p}} = t_{n}+m \Delta \tau $ and we let
\[f_{S,m}^{\left(  n\right)  }:=f_{S}\left(  t_{n}+\frac{m}{p}\Delta t\right), \mbox{   with     }
f_{S,m}^{\left(n\right)  }=P\,\mathbf{f}_{s,m}^{\left(  n\right)  }, \qquad 0\leq m\leq p.\]

In Algorithm~\ref{LTS-LF Galerkin FE Algorithm}, we list the full second-order LTS-LF Algorithm (\cite{DiazGrote09}, \cite[Alg. 1]{Grote_Mitkova})
for the sake of completeness. All computations in Steps 2 and 3 that involve the right-hand side $ \mathbf{f}_{S,m}^{\left(  n\right)}$ or the
stiffness matrix $\mathbf{A}$  only affect those degrees of freedom inside the
region of local refinement or directly adjacent to it. The successive updates of the
coarse unknowns involving $\mathbf{w}$ during sub-steps reduce
to a single standard LF step of size $\Delta t$ and, in fact, can be replaced by it.
In that sense, Algorithm~\ref{LTS-LF Galerkin FE Algorithm} yields a local time-stepping method. We remark that higher order
LTS-LF methods of arbitrarily high (even) accuracy were derived and implemented in \cite{DiazGrote09}.

%
%

\begin{algorithm} LTS-LF Galerkin FE Algorithm
\label{LTS-LF Galerkin FE Algorithm}
\qquad

\begin{enumerate}
\item Set $\mathbf{\tilde{u}}_{S,0}^{\left(  n\right)  }:=\mathbf{u}%
_{S}^{\left(  n\right)  }$ and compute $\mathbf{w}$ as
\[
\mathbf{w}=\mathbf{M}^{-1}\left(  \left(  \mathbf{M}-\mathbf{D}_{\mathcal{N}%
}\right)  \mathbf{f}_{S}^{\left(  n\right)  }-\mathbf{A}\left(  \mathbf{I}%
-\mathbf{M}^{-1}\mathbf{D}_{\mathcal{N}}\right)  \mathbf{u}_{S}^{\left(
n\right)  }\right)  .
\]

\item Compute
\[
\mathbf{\tilde{u}}_{S,1}^{\left(  n\right)  }=\mathbf{\tilde{u}}%
_{S,0}^{\left(  n\right)  }+\frac{1}{2}\left(  \frac{\Delta t}{p}\right)
^{2}\left(  \mathbf{w+M}^{-1}\left(  \mathbf{D}_{\mathcal{N}}\mathbf{f}%
_{S}^{\left(  n\right)  }-\mathbf{AM}^{-1}\mathbf{D}_{\mathcal{N}%
}\mathbf{\tilde{u}}_{S,0}^{\left(  n\right)  }\right)  \right)  .
\]

\item For $m=1,\ldots,p-1$, compute%

\begin{equation*}
\begin{aligned}
\mathbf{\tilde{u}}_{S,m+1}^{\left(  n\right)  }   = 2\mathbf{\tilde{u}}_{S,m}^{\left(  n\right)} -\mathbf{\tilde{u}}_{S,m-1}^{\left(  n\right)} & + \left(  \frac{\Delta t}{p}\right)  ^{2} \Bigg( \mathbf{w}  + \mathbf{M}^{-1}\bigg(
\frac{1}{2}\mathbf{D}_{\mathcal{N}}\left(  \mathbf{f}_{S,m}^{\left(  n\right)} +\mathbf{f}_{S,-m}^{\left(  n\right)  }\right) \\ & \qquad -  \mathbf{AM}^{-1} \mathbf{D}_{\mathcal{N}}\mathbf{\tilde{u}}_{S,m}^{\left(  n\right)} \bigg)\Bigg)
\end{aligned}
\end{equation*}

\item Compute%
\[
\mathbf{u}_{S}^{\left(  n+1\right)  }=-\mathbf{u}_{S}^{\left(  n-1\right)
}+2\mathbf{\tilde{u}}_{S,p}^{\left(  n\right)  }.
\]

\end{enumerate}
\end{algorithm}

Like the standard leap-frog method (without local time-stepping),
the LTS-LF Algorithm requires in principle the solution of a
linear system involving ${\mathbf M}$ at every time-step.
Although the mass matrix is sparse, positive definite, and well-conditioned
so that solving linear systems with this matrix is relatively cheap,
this computational effort is commonly avoided by using
either mass-lumping techniques \cite{CJRT01,Mul01}, spectral elements \cite{CHQZ06,KS05}
or discontinuous Galerkin finite elements \cite{ABCM02,GSS06}. The resulting LTS-LF scheme
is then fully explicit.

In \cite{DiazGrote09}, the above LTS-LF Algorithm was rewritten
in \textquotedblleft leap-frog manner\textquotedblright\ by introducing the
perturbed bilinear form $a_{p}:S\times S\rightarrow\mathbb{R}$:
\begin{equation}
a_{p}\left(  u,v\right)  :=a\left(  u,v\right)  -\frac{2}{p^{2}}\sum
_{j=1}^{p-1}\alpha_{j}^{p}\left(  \frac{\Delta t}{p}\right)  ^{2j}a\left(
\left(  R_{\mathcal{N}}A_{S}\right)  ^{j}u,v\right)  \qquad\forall u,v\in
S\label{defapv}%
\end{equation}
with associated operator%
\begin{equation}
A_{S,p}:S\rightarrow S,\qquad A_{S,p}:=A_{S}-\frac{2}{p^{2}}\sum_{j=1}%
^{p-1}\alpha_{j}^{p}\left(  \frac{\Delta t}{p}\right)  ^{2j}A_{S}\left(
R_{\mathcal{N}}A_{S}\right)  ^{j}.\label{defASp}%
\end{equation}
Here the constants $\alpha_{j}^{m}$, $j =1, \dots, m-1$ are recursively defined for $m\geq 2$ by%
\begin{equation}%
\begin{array}
[c]{ccc}%
\alpha_{1}^{2}=\frac{1}{2} & \alpha_{1}^{3}=3, & \alpha_{2}^{3}=-\frac{1}{2}\\ \\
\multicolumn{1}{r}{\alpha_{1}^{m+1}} & \multicolumn{1}{l}{=\frac{m^{2}}%
{2}+2\alpha_{1}^{m}-\alpha_{1}^{m-1},} & \multicolumn{1}{l}{}\\
\multicolumn{1}{r}{\alpha_{j}^{m+1}} & \multicolumn{1}{l}{=2\alpha_{j}%
^{m}-\alpha_{j}^{m-1}-\alpha_{j-1}^{m},} & \multicolumn{1}{l}{j=2,\ldots
,m-2,}\\
\multicolumn{1}{r}{\alpha_{m-1}^{m+1}} & \multicolumn{1}{l}{=2\alpha_{m-1}%
^{m}-\alpha_{m-2}^{m},} & \multicolumn{1}{l}{}\\
\multicolumn{1}{r}{\alpha_{m}^{m+1}} & \multicolumn{1}{l}{=-\alpha_{m-1}^{m}.}
& \multicolumn{1}{l}{}%
\end{array}
\label{rekalpha}%
\end{equation}
Then the LTS-LF scheme (Algorithm~\ref{LTS-LF Galerkin FE Algorithm}) is equivalent to
\begin{equation}%
\begin{array}
[c]{cc}%
\left(  u_{S}^{\left(  n+1\right)  }-2u_{S}^{\left(  n\right)  }%
+u_{S}^{\left(  n-1\right)  },w\right)  +\Delta t^{2}a_{p}\left(
u_{S}^{\left(  n\right)  },w\right)  =\Delta t^{2}\left(  f_{S}^{\left(
n\right)  },w\right)   & \forall w\in S,\\
\multicolumn{1}{r}{\left.
\begin{array}
[c]{l}%
\left(  u_{S}^{\left(  0\right)  },w\right)  =\left(  u_{0},w\right)  \\
\\
\left(  u_{S}^{\left(  1\right)  },w\right)  =\left(  u_{0},w\right)  +\Delta
t\left(  v_{0},w\right)  +\frac{\Delta t^{2}}{2}\left(  f_{S}^{\left(
0\right)  }\left(  w\right)  -a\left(  u_{0},w\right)  \right)
\end{array}
\right\}  } & \multicolumn{1}{r}{\forall w\in S.}%
\end{array}
\label{leap_frog_lts_fd}%
\end{equation}
Neither the equivalent formulation \eqref{leap_frog_lts_fd} nor the constants $\alpha_{j}^{m}$ are ever used in practice but only for the purpose of analysis; in fact, the constants $\alpha_{j}^{m}$ do not
appear in Algorithm 1.

\begin{remark}
In (\ref{leap_frog_lts_fd}) the term $a\left(  u_{0},w\right)  $ in
the third equation could be replaced by $a_{p}\left(  u_{0},w\right)  $ which allows for
local time-stepping already during the very first time-step. In that case,
the analysis below also applies but requires a minor change,
namely, replacing $A_{S}$ by $A_{S,p}$ in (\ref{initerr1}) and (\ref{initerr2}).
This modification neither affects the stability nor the convergence rate of the overall LTS-LF scheme.
%
\end{remark}

\section{Stability and Convergence Analysis\label{SecStabCons}}

\subsection{Estimates of the Bilinearform\label{SecSetting}}

The following equivalence of the continuous $L^{2}\left(  \Omega\right)  $-
and mesh-dependent norm is well known.

\begin{lemma}
\label{Satz9.1}$\left\Vert \cdot\right\Vert _{\mathcal{T}}$ and $\left\Vert
\cdot\right\Vert $ are equivalent norms on $S_{\mathcal{T}}^{m}$. The
constants $c_{\operatorname*{eq}}$, $C_{\operatorname*{eq}}$ in the
equivalence estimates
\[
c_{\operatorname*{eq}}\left\Vert u\right\Vert _{\mathcal{T}}\leq\left\Vert
u\right\Vert \leq C_{\operatorname*{eq}}\left\Vert u\right\Vert _{\mathcal{T}%
}\qquad\forall u\in S_{\mathcal{T}}^{m}%
\]
only depend on the polynomial degree $m$ and the shape regularity constant
$\gamma\left(  \mathcal{T}\right)  $.
\end{lemma}

It is also well known that the functions in $S_{\mathcal{T}}^{m}$ satisfy an
inverse inequality (for a proof we refer, e.g., \cite[(3.2.33) with $m=1$,
$q=r=2$, $l=0$, $n=d$.]{CiarletPb}\footnote{There is a misprint in this
reference: $m-1$ should be replaced by $m-\ell$, see also \cite[(4.5.3)
Lemma]{scottbrenner3}.}).

\begin{lemma}
\label{Leminvineq}There exists a constant $C_{\operatorname*{inv}}>0$, which
only depends on $\gamma\left(  \mathcal{T}\right)  $ and $m$, such that for
all $\tau\in\mathcal{T}$%
\begin{equation}
\left\Vert \nabla u\right\Vert _{L^{2}\left(  \tau\right)  }\leq
C_{\operatorname*{inv}}h_{\tau}^{-1}\left\Vert u\right\Vert _{L^{2}\left(
\tau\right)  },\qquad\forall u\in\mathcal{S}_{\mathcal{T}}^{m}.
\label{defcinv}%
\end{equation}
The global versions of the inverse inequality involves also the
quasi-uniformity constant%
\begin{equation}
\left\Vert \nabla u\right\Vert \leq C_{\operatorname*{inv}}%
C_{\operatorname*{qu}}h^{-1}\left\Vert u\right\Vert \quad\text{and\quad
}\left\Vert u\right\Vert _{H^{1}\left(  \Omega\right)  }\leq\sqrt
{1+C_{\operatorname*{inv}}^{2}C_{\operatorname*{qu}}^{2}h^{-2}}\,\left\Vert
u\right\Vert \label{globinveq}%
\end{equation}
for all $u\in S_{\mathcal{T}}^{m}$.
\end{lemma}

In the next step, we will estimate $\left\Vert A_{S}u\right\Vert $ in terms of
$\left\Vert u\right\Vert _{H^{1}\left(  \Omega\right)  }$.

\begin{lemma}
\label{LemAsu}It holds%
\begin{equation}
\left\Vert A_{S}u\right\Vert \leq C_{\operatorname*{cont}}\sqrt
{1+C_{\operatorname*{inv}}^{2}C_{\operatorname*{qu}}^{2}h^{-2}}\left\Vert
u\right\Vert _{H^{1}\left(  \Omega\right)  }\qquad\forall u\in S.
\label{estAs}%
\end{equation}

\end{lemma}

%

\proof
Since $A_{S}$ is a self-adjoint, positive operator there exists an orthonormal
system $\left(  \eta_{\nu}\right)  _{\nu=1}^{M}$ such that%
\[
A_{S}\eta_{\nu}=\lambda_{\nu}\eta_{\nu}%
\]
and%
\[
\left(  \eta_{\nu},\eta_{\mu}\right)  =\delta_{\nu,\mu}%
\]
where $M:=\dim S$. Hence, every function $v\in S$ has a representation%
\[
v=\sum_{\nu=1}^{M}c_{\nu}\eta_{\nu}.
\]

For $s\in\mathbb{R}$ we define the norm on $S$%
\[
\left\vert
\kern-.1em%
\left\vert
\kern-.1em%
\left\vert v\right\vert
\kern-.1em%
\right\vert
\kern-.1em%
\right\vert _{s}:=\left\{  \sum_{\mu=1}^{M}\lambda_{\mu}^{s}c_{\mu}%
^{2}\right\}  ^{1/2}.
\]
It is obvious that for all $v\in S$, it holds%
\begin{align*}
\left\vert
\kern-.1em%
\left\vert
\kern-.1em%
\left\vert v\right\vert
\kern-.1em%
\right\vert
\kern-.1em%
\right\vert _{0}  &  =\left\Vert v\right\Vert ,\\
\left\vert
\kern-.1em%
\left\vert
\kern-.1em%
\left\vert v\right\vert
\kern-.1em%
\right\vert
\kern-.1em%
\right\vert _{1}  &  =a\left(  v,v\right)  ^{1/2}\lesseqgtr\left\{
\begin{array}
[c]{c}%
C_{\operatorname*{cont}}^{1/2}\left\Vert v\right\Vert _{H^{1}\left(
\Omega\right)  },\\
c_{\operatorname*{coer}}^{1/2}\left\Vert v\right\Vert _{H^{1}\left(
\Omega\right)  }.
\end{array}
\right.
\end{align*}
Note that
%

\[
{\left\vert
\kern-.1em%
\left\vert
\kern-.1em%
\left\vert v\right\vert
\kern-.1em%
\right\vert
\kern-.1em%
\right\vert _{2}^{2}:=\sum_{\mu=1}^{M}\lambda_{\mu}^{2}c_{\mu}^{2}=\sum
_{\mu,\nu=1}^{M}\lambda_{\mu}c_{\mu}\lambda_{\nu}c_{\nu}\left(  \eta_{\mu
},\eta_{\nu}\right)  =\left(  A_{S}v,A_{S}v\right)  .}%
\]

We assume that the eigenvalues $\lambda_{\nu}$ are ordered increasingly. From
Lemma \ref{Leminvineq} we conclude that{%
\[
\lambda_{M}:=\max_{u\in S\backslash\left\{  0\right\}  }\frac{a\left(
u,u\right)  }{\left(  u,u\right)  }\leq C_{\operatorname*{cont}}\max_{u\in
S\backslash\left\{  0\right\}  }\frac{\left\Vert u\right\Vert _{H^{1}\left(
\Omega\right)  }^{2}}{\left\Vert u\right\Vert ^{2}}\overset
{\text{(\ref{defcinv})}}{\leq}C_{\operatorname*{cont}}\left(
1+C_{\operatorname*{inv}}^{2}C_{\operatorname*{qu}}^{2}h^{-2}\right)
\]
holds. Hence,%
\[
\left\Vert A_{S}v\right\Vert ^{2}\leq C_{\operatorname*{cont}}\left(
1+C_{\operatorname*{inv}}^{2}C_{\operatorname*{qu}}^{2}h^{-2}\right)
\sum_{\mu=1}^{M}\lambda_{\mu}c_{\mu}^{2}\leq C_{\operatorname*{cont}}%
^{2}\left(  1+C_{\operatorname*{inv}}^{2}C_{\operatorname*{qu}}^{2}%
h^{-2}\right)  \left\Vert v\right\Vert _{H^{1}\left(  \Omega\right)  }^{2}.
\]%

\smallskip

\endproof
}

\medskip

Next, we will estimate the bilinear form $a_{p}\left(  \cdot,\cdot\right)  $.

\begin{lemma}
\label{LemEstRN}The operator $R_{\mathcal{N}}$ as in (\ref{defrn1}) has
bounded $L^{2}\left(  \Omega\right)  $ norm:%
\begin{equation}
\left\Vert R_{\mathcal{N}}u\right\Vert \leq c_{\operatorname*{eq}}%
^{-2}\left\Vert u\right\Vert \qquad\forall u\in\mathcal{S}_{\mathcal{T}}^{m}.
\label{Rnest}%
\end{equation}
For $u\in\mathcal{S}_{\mathcal{T}}^{m}$ it holds%
\begin{equation}
\left\Vert   R_{\mathcal{N}}A_{S}u\right\Vert \leq
\frac{C_{\operatorname*{cont}}}{c_{\operatorname*{eq}}^{2}}
\left(  1+\frac{C_{\operatorname*{inv}}^{2}C_{\operatorname*{qu}}^{2}}{h^{2}%
}\right)\left\Vert u\right\Vert .
  \label{estpowerop}%
\end{equation}

\end{lemma}

%

\begin{proof}
Let $u=P\mathbf{u}$ and $v=P\mathbf{v}$ with $\mathbf{u}=\left(  u_{z}\right)
_{z\in\Sigma_{\mathcal{T}}^{m}}$, $\mathbf{v}=\left(  v_{z}\right)
_{z\in\Sigma_{\mathcal{T}}^{m}}$. We employ%
\[
\left(  R_{\mathcal{N}}u,v\right)  =\left\langle \mathbf{D}_{\mathcal{N}%
}\mathbf{u},\mathbf{v}\right\rangle =\sum_{z\in\mathcal{N}}d_{z}u_{z}v_{z}.
\]
Hence%
\begin{align*}
\left\Vert R_{\mathcal{N}}u\right\Vert  &  =\sup_{v\in\mathcal{S}%
_{\mathcal{T}}^{m}\backslash\left\{  0\right\}  }\frac{\sum_{z\in\mathcal{N}%
}d_{z}u_{z}v_{z}}{\left\Vert v\right\Vert }\leq\sup_{v\in\mathcal{S}%
_{\mathcal{T}}^{m}\backslash\left\{  0\right\}  }\frac{\sum_{z\in\mathcal{N}%
}d_{z}\left\vert u_{z}\right\vert \left\vert v_{z}\right\vert }{\left\Vert
v\right\Vert }\\
&  \leq\sup_{v\in\mathcal{S}_{\mathcal{T}}^{m}\backslash\left\{  0\right\}
}\frac{\left\langle \mathbf{D}_{\Sigma_{\mathcal{T}}^{m}}\mathbf{u}%
,\mathbf{u}\right\rangle ^{1/2}\left\langle \mathbf{D}_{\Sigma_{\mathcal{T}%
}^{m}}\mathbf{v},\mathbf{v}\right\rangle ^{1/2}}{\left\Vert v\right\Vert
}=\left\Vert u\right\Vert _{\mathcal{T}}\sup_{v\in\mathcal{S}_{\mathcal{T}%
}^{m}\backslash\left\{  0\right\}  }\frac{\left\Vert v\right\Vert
_{\mathcal{T}}}{\left\Vert v\right\Vert }\\
&  \leq c_{\operatorname*{eq}}^{-2}\left\Vert u\right\Vert .
\end{align*}
For the second estimate we employ (\ref{estAs}) and (\ref{globinveq}) to
obtain%
\begin{equation}
\left\Vert R_{\mathcal{N}}A_{S}u\right\Vert \leq c_{\operatorname*{eq}}%
^{-2}\left\Vert A_{S}u\right\Vert \leq\frac{C_{\operatorname*{cont}}%
}{c_{\operatorname*{eq}}^{2}}
\left(  1+C_{\operatorname*{inv}}^{2}C_{\operatorname*{qu}}^{2}h^{-2}\right)
\left\Vert u\right\Vert \label{RnAEst}%
\end{equation}
for all $u\in\mathcal{S}_{\mathcal{T}}^{m}$.
\end{proof}

\begin{lemma}
{\label{LemapH1est}Let the bilinear form $a\left(  \cdot,\cdot\right)  $
satisfy (\ref{wellposed}) and let the \emph{CFL condition}
\begin{equation}
C_{\operatorname*{cont}}\Delta t^{2}\left(  1+\frac{C_{\operatorname*{inv}%
}^{2}C_{\operatorname*{qu}}^{2}}{h^{2}}\right)  \leq\min\left\{
6c_{\operatorname*{eq}}^{2}\left(  \frac{c_{\operatorname*{coer}}%
}{C_{\operatorname*{cont}}}\right)  ^{3/2},\frac{4C_{\operatorname*{cont}}%
}{\max\{C_{\operatorname*{cont}},3\}}\right\}
\label{CFL}%
\end{equation}
hold. \newline Then, the bilinear form $a_{p}\left(  \cdot,\cdot\right)  $ is
continuous,
\[
\left\vert a_{p}\left(  u,v\right)  \right\vert \leq C_{\operatorname*{cont}%
}\left(  1+\sqrt{\frac{C_{\operatorname*{cont}}}{c_{\operatorname*{coer}}}%
}\frac{\kappa}{12}\right)  \left\Vert u\right\Vert _{H^{1}\left(
\Omega\right)  }\left\Vert v\right\Vert _{H^{1}\left(  \Omega\right)  }%
\]

%
with
\begin{equation}
\kappa:=\left(  \frac{C_{\operatorname*{cont}}}{c_{\operatorname*{eq}}^{2}%
}\right)  \Delta t^{2}\left(  1+\frac{C_{\operatorname*{inv}}^{2}%
C_{\operatorname*{qu}}^{2}}{h^{2}}\right)  , \label{defz}%
\end{equation}
and symmetric, $a_{p}\left(  u,v\right)  =a_{p}\left(  v,u\right)  $ for all
$u,v\in S.$ Moreover, for any $f\in L^{2}\left(  \Omega\right)  $, the
problem: Find $u\in S$ such that%
\[
a_{p}\left(  u,q\right)  =\left(  f,q\right)  \qquad\forall q\in S
\]
has a unique solution, which satisfies%
\[
\left\Vert u\right\Vert _{H^{1}\left(  \Omega\right)  }\leq\frac
{2}{c_{\operatorname*{coer}}}\left\Vert f\right\Vert .
\]
}
\end{lemma}

\begin{remark}
In (\ref{CFL}) the condition on the time-step $\Delta t$ implies that $\Delta
t$ is essentially proportional to $h$ and inversely proportional to
$\sqrt{C_{\operatorname*{cont}}}$, as $c_{\operatorname*{coer}} \leq
C_{\operatorname*{cont}}$. Hence (\ref{CFL}) corresponds to a genuine CFL
condition since $\sqrt{C_{\operatorname*{cont}}}$ usually corresponds to the
maximal (physical) wave speed.
\end{remark}

\textbf{Proof of Lemma \ref{LemapH1est}.} {If $p=1$, the two bilinear forms
$a_{p}$ and $a$ coincide and the result trivially follows. Thus, we now assume
that $p\geq2$.\newline} \textbf{a) Continuity. }Let $u,v\in S$ and
\begin{equation}
w:=u-\frac{2}{p^{2}}\sum_{j=1}^{p-1}\alpha_{j}^{p}\left(  \frac{\Delta t}%
{p}\right)  ^{2j}\left(  R_{\mathcal{N}}A_{S}\right)  ^{j}u. \label{defw}%
\end{equation}
Then, by definition of $a_{p}$ and continuity of $a$, we have%
\[
\left\vert a_{p}\left(  u,v\right)  \right\vert =\left\vert a\left(
w,v\right)  \right\vert \leq C_{\operatorname*{cont}}\left\Vert w\right\Vert
_{H^{1}\left(  \Omega\right)  }\left\Vert v\right\Vert _{H^{1}\left(
\Omega\right)  }.
\]
By applying the triangle
inequality to \eqref{defw} we obtain
\begin{align*}
\left\Vert w\right\Vert _{H^{1}\left(  \Omega\right)  } &  \leq\left\Vert
u\right\Vert _{H^{1}\left(  \Omega\right)  }+\frac{2}{p^{2}}\left\Vert
\sum_{j=1}^{p-1}\alpha_{j}^{p}\left(  \frac{\Delta t}{p}\right)  ^{2j}\left(
R_{\mathcal{N}}A_{S}\right)  ^{j}u\right\Vert _{H^{1}\left(  \Omega\right)
}\\
&  \leq\left\Vert u\right\Vert _{H^{1}\left(  \Omega\right)  }+\frac{2}{p^{2}%
}\left\Vert A_{S}^{-1/2}\sum_{j=1}^{p-1}\alpha_{j}^{p}\left(  \frac{\Delta
t}{p}\right)  ^{2j}\left(  A_{S}^{1/2}R_{\mathcal{N}}A_{S}^{1/2}\right)
^{j}A_{S}^{1/2}u\right\Vert _{H^{1}\left(  \Omega\right)  }.
\end{align*}
From (\ref{wellposed}), it follows that%
\[
\left\Vert A_{S}^{-1/2}u\right\Vert _{H^{1}\left(  \Omega\right)  }^{2}%
\leq\frac{1}{c_{\operatorname*{coer}}}\left\Vert u\right\Vert ^{2}%
\quad\text{and\quad}\left\Vert A_{S}^{1/2}u\right\Vert ^{2}\leq
C_{\operatorname*{cont}}\left\Vert u\right\Vert _{H^{1}\left(  \Omega\right)
}^{2}\quad\forall u\in S.
\]
Hence,%
\begin{equation}
\left\Vert w\right\Vert _{H^{1}\left(  \Omega\right)  }\leq\left(
1+C_{p}\sqrt{\frac{C_{\operatorname*{cont}}}{c_{\operatorname*{coer}}}%
}\right)  \left\Vert u\right\Vert _{H^{1}\left(  \Omega\right)  }%
.\label{wCpest}%
\end{equation}
with%
\[
C_{p}:=\sup_{v\in S\backslash\left\{  0\right\}  }\frac{2}{p^{2}}\left\Vert
\sum_{j=1}^{p-1}\alpha_{j}^{p}\left(  \frac{\Delta t}{p}\right)  ^{2j}\left(
A_{S}^{1/2}R_{\mathcal{N}}A_{S}^{1/2}\right)  ^{j}v\right\Vert \bigg/ \left\Vert
v\right\Vert .
\]

The operator $A_{S}^{1/2}R_{\mathcal{N}}A_{S}^{1/2}$ is self-adjoint with
respect to the $L^{2}\left(  \Omega\right)  $ scalar product and positive
semi-definite. It is well-known that under these conditions we have%
\[
C_{p}=\max_{\lambda\in\sigma\left(  A_{S}^{1/2}R_{\mathcal{N}}A_{S}^{1/2}\right)
}\frac{2}{p^{2}}\left\vert \sum_{j=1}^{p-1}\alpha_{j}^{p}\left(  \frac{\Delta
t}{p}\right)  ^{2j}\lambda^{j}\right\vert .
\]
From (\ref{estpowerop}) we conclude that the spectrum $\sigma\left(
A_{S}^{1/2}R_{\mathcal{N}}A_{S}^{1/2}\right)  $ is contained in the interval
$\left[  0,\frac{C_{\operatorname*{cont}}}{c_{\operatorname*{eq}}^{2}%
}\left(  1+\frac{C_{\operatorname*{inv}}^{2}C_{\operatorname*{qu}}^{2}}{h^{2}%
}\right)  \right]  $ so that%
\[
C_{p}\leq\sup_{0\leq x\leq\kappa}\frac{2}{p^{2}}\left\vert \sum_{j=1}%
^{p-1}\alpha_{j}^{p}\left(  \frac{x}{p^{2}}\right)  ^{j}\right\vert
\]
with $\kappa$ as in (\ref{defz}). The CFL condition (\ref{CFL}), together
with the continuity and the coercivity of $a$ and $p\geq2$, implies $\kappa
\in\left[  0,4p^{2}\right]  $. Thus, Lemma \ref{LemTschebaschev}
(Appendix) implies
\begin{equation}
C_p \leq \frac{\kappa}{12},
\label{perturb_est1}%
\end{equation}
which we insert in \eqref{wCpest} to obtain%
\[
\left\Vert w\right\Vert _{H^{1}\left(  \Omega\right)  }\leq\left(
1+\frac{\kappa}{12}\sqrt{\frac{C_{\operatorname*{cont}}}%
{c_{\operatorname*{coer}}}}\right)  \left\Vert u\right\Vert _{H^{1}\left(
\Omega\right)  }.
\]

\textbf{b) Symmetry. }This follows since $A_{S}$, $R_{\mathcal{N}}$ are
self-adjoint with respect to the $L^{2}\left(  \Omega\right)  $ scalar product.

\textbf{c) Coercivity. }Note that the problem: Find $u\in S$ such that%
\[
a_{p}\left(  u,q\right)  =\left(  f,q\right)  \quad\forall q\in S
\]
can be solved in two steps: Find $w\in S$ such that%
\begin{equation}
a\left(  w,q\right)  =\left(  f,q\right)  \quad\forall q\in S. \label{LMS}%
\end{equation}
Then $u$ is the solution of%
\[
\left(  I-\frac{2}{p^{2}}\sum_{j=1}^{p-1}\alpha_{j}^{p}\left(  \frac{\Delta
t}{p}\right)  ^{2j}\left(  R_{\mathcal{N}}A_{S}\right)  ^{j}\right)  u=w.
\]
By the similar arguments as in the first part of this proof, one concludes that
the CFL-condition (\ref{CFL}) implies%
\begin{equation}
\left\Vert \frac{2}{p^{2}}\sum_{j=1}^{p-1}\alpha_{j}^{p}\left(  \frac{\Delta
t}{p}\right)  ^{2j}\left(  R_{\mathcal{N}}A_{S}\right)  ^{j}q\right\Vert
_{H^{1}\left(  \Omega\right)  }\leq\frac{1}{2}\left\Vert q\right\Vert
_{H^{1}\left(  \Omega\right)  }\quad\forall q\in S \label{ImASP}%
\end{equation}
so that%
\[
\left\Vert u\right\Vert _{H^{1}\left(  \Omega\right)  }\leq2\left\Vert
w\right\Vert _{H^{1}\left(  \Omega\right)  }.
\]
The well-posedness of problem (\ref{LMS}) follows from the Lax-Milgram lemma
as well as the estimate%
\[
\left\Vert w\right\Vert _{H^{1}\left(  \Omega\right)  }\leq\frac
{1}{c_{\operatorname*{coer}}}\left\Vert f\right\Vert .
\]%
\endproof

\begin{corollary}
\label{Coreigsys}The bilinear form $a_{p}\left(  u,v\right)  $ is symmetric,
continuous and coercive. Hence, there exists an $L^{2}\left(  \Omega\right)
$-orthonormal eigensystem $\left(  \lambda_{S,p,k},\eta_{S,p,k}\right)
_{k=1}^{M}$ for $a_{p}\left(  \cdot,\cdot\right)  $, i.e.,%
\[%
\begin{array}
[c]{cl}%
a_{p}\left(  \eta_{S,p,k},v\right)  =\lambda_{S,p,k}\left(  \eta
_{S,p,k},v\right)  & \forall v\in S,\\
\left(  \eta_{S,p,k},\eta_{S,p,\ell}\right)  =\delta_{k,\ell} & \forall
k,\ell\in\left\{  1,\ldots,M\right\}  ,
\end{array}
\]
with real and positive eigenvalues $\lambda_{S,p,k}>0$. Let the CFL condition
(\ref{CFL}) be satisfied. Then, the smallest and largest eigenvalue satisfy
\[
\lambda_{p}^{\min}\geq\frac{c_{\operatorname*{coer}}}{2}\quad\text{and\quad
}\lambda_{p}^{\max}\leq\frac{3}{2}C_{\operatorname*{cont}}\left(
1+C_{\operatorname*{inv}}^{2}C_{\operatorname*{qu}}^{2}h^{-2}\right)  .
\]

\end{corollary}

%

\proof
We start with the smallest eigenvalue. It holds%
\begin{align*}
\left\vert a\left(  \frac{2}{p^{2}}\sum_{j=1}^{p-1}\alpha_{j}^{p}\left(
\frac{\Delta t}{p}\right)  ^{2j}\!\!\left(  R_{\mathcal{N}}A_{S}\right)
^{j}v,v\right)  \!\!\right\vert  &  \leq C_{\operatorname*{cont}}\left\Vert
\frac{2}{p^{2}}\sum_{j=1}^{p-1}\alpha_{j}^{p}\left(  \frac{\Delta t}%
{p}\right)  ^{2j}\!\!\left(  R_{\mathcal{N}}A_{S}\right)  ^{j}\!\!v\right\Vert
_{H^{1}\left(  \Omega\right)  }\!\!\!\!\left\Vert v\right\Vert _{H^{1}\left(
\Omega\right)  }\\
&  \overset{\text{(\ref{perturb_est1})}}{\leq}C_{\operatorname*{cont}}%
\sqrt{\frac{C_{\operatorname*{cont}}}{c_{\operatorname*{coer}}}}\frac{\kappa
}{12}\left\Vert v\right\Vert _{H^{1}\left(  \Omega\right)  }^{2}%
\end{align*}
with $\kappa$ as in (\ref{defz}). Hence,%
\begin{align*}
a_{p}\left(  v,v\right)   &  =a\left(  v,v\right)  -a\left(  \frac{2}{p^{2}%
}\sum_{j=1}^{p-1}\alpha_{j}^{p}\left(  \frac{\Delta t}{p}\right)  ^{2j}\left(
R_{\mathcal{N}}A_{S}\right)  ^{j}v,v\right) \\
&  \geq\left(  c_{\operatorname*{coer}}-C_{\operatorname*{cont}}\sqrt
{\frac{C_{\operatorname*{cont}}}{c_{\operatorname*{coer}}}}\frac{\kappa}%
{12}\right)  \left\Vert v\right\Vert _{H^{1}\left(  \Omega\right)  }^{2}.
\end{align*}
The CFL condition (\ref{CFL}) implies%
\begin{subequations}
\label{h1equiconst}
\end{subequations}%
\begin{equation}
a_{p}\left(  v,v\right)  \geq\frac{c_{\operatorname*{coer}}}{2}\left\Vert
v\right\Vert _{H^{1}\left(  \Omega\right)  }^{2}\geq\frac
{c_{\operatorname*{coer}}}{2}\left\Vert v\right\Vert ^{2} \tag{%
\ref{h1equiconst}%
a}\label{h1equiconsta}%
\end{equation}
which yields the lower bound on the smallest eigenvalue $\lambda_{p}^{\min}$.

For the largest eigenvalue $\lambda_{p}^{\max}$, we get by using the CFL
condition and (\ref{globinveq}) that%
\begin{equation}
\left\vert a_{p}\left(  v,v\right)  \right\vert \leq\frac{3}{2}%
C_{\operatorname*{cont}}\left\Vert v\right\Vert _{H^{1}\left(  \Omega\right)
}^{2}\leq\frac{3}{2}C_{\operatorname*{cont}}\left(  1+C_{\operatorname*{inv}%
}^{2}C_{\operatorname*{qu}}^{2}h^{-2}\right)  \left\Vert v\right\Vert ^{2},
\tag{%
\ref{h1equiconst}%
b}\label{h1equiconstb}%
\end{equation}
from which the upper bound on $\lambda_{p}^{\max}$ follows.%
\endproof

\begin{corollary}
\label{CorAspL2}Let the assumptions of Lemma \ref{LemapH1est} be satisfied.
Then%
\[
\left\Vert A_{S,p}^{-1}w\right\Vert \leq\frac{2}{c_{\operatorname*{coer}}%
}\left\Vert w\right\Vert \qquad\forall w\in S,
\]
uniformly in $p$.
\end{corollary}%

\proof
We write%
\[
A_{S,p}^{-1}=\left(  I_{S}-\frac{2}{p^{2}}\sum_{j=1}^{p-1}\alpha_{j}%
^{p}\left(  \frac{\Delta t}{p}\right)  ^{2j}\left(  R_{\mathcal{N}}%
A_{S}\right)  ^{j}\right)  ^{-1}A_{S}^{-1}.
\]
Note that for all $w\in S$ it holds%
\[
\left\Vert \frac{2}{p^{2}}\sum_{j=1}^{p-1}\alpha_{j}^{p}\left(  \frac{\Delta
t}{p}\right)  ^{2j}\left(  R_{\mathcal{N}}A_{S}\right)  ^{j}w\right\Vert
=\left\Vert R_{\mathcal{N}}^{1/2}\frac{2}{p^{2}}\sum_{j=1}^{p-1}\alpha_{j}%
^{p}\left(  \frac{\left(  \Delta t\right)  ^{2}}{p^{2}}R_{\mathcal{N}}%
^{1/2}A_{S}R_{\mathcal{N}}^{1/2}\right)  ^{j-1}\left(  \frac{\Delta t}%
{p}\right)  ^{2}\left(  R_{\mathcal{N}}^{1/2}A_{S}\right)  w\right\Vert .
\]
Since $R_{\mathcal{N}}$  is symmetric, positive semi-definite (see Remark 3), we
infer from (\ref{Rnest}) that $\left\Vert R_{\mathcal{N}}^{1/2}v\right\Vert
\leq c_{\operatorname*{eq}}^{-1}\left\Vert v\right\Vert $ holds for all $v\in
S$. From Lemmas \ref{Leminvineq} and \ref{LemAsu} we obtain for all $v\in S$%
\begin{align*}
\left\Vert \left(  R_{\mathcal{N}}^{1/2}A_{S}\right)  v\right\Vert  & \leq
c_{\operatorname*{eq}}^{-1}\left\Vert A_{S}v\right\Vert \\
& \leq\frac{C_{\operatorname*{cont}}}{c_{\operatorname*{eq}}}\sqrt
{1+C_{\operatorname*{inv}}^{2}C_{\operatorname*{qu}}^{2}h^{-2}}\left\Vert
v\right\Vert _{H^{1}\left(  \Omega\right)  }\leq\frac{C_{\operatorname*{cont}%
}}{c_{\operatorname*{eq}}}\left(  1+C_{\operatorname*{inv}}^{2}%
C_{\operatorname*{qu}}^{2}h^{-2}\right)  \left\Vert v\right\Vert .
\end{align*}
Thus, we argue as for (\ref{wCpest}) and get%
\[
\left\Vert \frac{2}{p^{2}}\sum_{j=1}^{p-1}\alpha_{j}^{p}\left(  \frac{\Delta
t}{p}\right)  ^{2j}\left(  R_{\mathcal{N}}A_{S}\right)  ^{j}w\right\Vert \leq
C_{p}^{\prime}\frac{C_{\operatorname*{cont}}}{c_{\operatorname*{eq}}^{2}%
}\left(  \frac{\Delta t}{p}\right)  ^{2}\left(  1+C_{\operatorname*{inv}}%
^{2}C_{\operatorname*{qu}}^{2}h^{-2}\right)  \left\Vert w\right\Vert
\]
with%
\[
C_{p}^{\prime}:=\max_{\lambda\in\sigma\left(  R_{\mathcal{N}}^{1/2}A_{S}%
R_{\mathcal{N}}^{1/2}\right)  }\frac{2}{p^{2}}\left\vert \sum_{j=1}%
^{p-1}\alpha_{j}^{p}\left(  \frac{\left(  \Delta t\right)  ^{2}\lambda}{p^{2}%
}\right)  ^{j-1}\right\vert .
\]
From Lemma \ref{LemTschebaschev} we conclude that $C_{p}^{\prime}\leq
(p^{2}-1)/12\leq p^{2}/12$ so that (\ref{CFL})
implies%
\[
\left\Vert \frac{2}{p^{2}}\sum_{j=1}^{p-1}\alpha_{j}^{p}\left(  \frac{\Delta
t}{p}\right)  ^{2j}\left(  R_{\mathcal{N}}A_{S}\right)  ^{j}w\right\Vert
\leq\frac{C_{\operatorname*{cont}}}{12 \,c_{\operatorname*{eq}}^{2}}\left(  \Delta
t\right)  ^{2}\left(  1+C_{\operatorname*{inv}}^{2}C_{\operatorname*{qu}}%
^{2}h^{-2}\right)  \left\Vert w\right\Vert \leq\frac{1}{2}\left\Vert
w\right\Vert .
\]
Thus, we have proved%
\begin{equation}
\left\Vert \left(  I_{S}-\frac{2}{p^{2}}\sum_{j=1}^{p-1}\alpha_{j}^{p}\left(
\frac{\Delta t}{p}\right)  ^{2j}\left(  R_{\mathcal{N}}A_{S}\right)
^{j}\right)  ^{-1}w\right\Vert \leq2\left\Vert w\right\Vert \qquad\forall w\in
S\text{.}\label{L2Ap}%
\end{equation}
From (\ref{wellposedc}) we conclude that%
\[
\left\Vert A_{S}^{-1}w\right\Vert \leq c_{\operatorname*{coer}}^{-1}\left\Vert
w\right\Vert \qquad\forall w\in S,
\]
which together with \eqref{L2Ap} leads to the assertion.
\endproof

\subsection{Error equation and estimates}

{To derive a priori error estimates for the LTS/FE-Galerkin solution of
(\ref{leap_frog_lts_fd}), we first introduce the new function
\begin{equation}
v_{S}^{\left(  n+1/2\right)  }:=\frac{u_{S}^{\left(  n+1\right)  }%
-u_{S}^{\left(  n\right)  }}{\Delta t},\label{defvS}%
\end{equation}
and rewrite (\ref{leap_frog_lts_fd}) as a one-step method
\begin{equation}%
\begin{split}
\left(  v_{S}^{\left(  n+1/2\right)  },q\right)  &=  \left(  v_{S}^{\left(
n-1/2\right)  },q\right)  -\Delta ta_{p}\left(  u_{S}^{\left(  n\right)
},q\right)  +\Delta tF^{\left(  n\right)  }\left(  q\right)   \; \forall q\in
S,\\
-\Delta t\left(  v_{S}^{\left(  n+1/2\right)  },r\right)  +\left(
u_{S}^{\left(  n+1\right)  },r\right)  &=  \left(  u_{S}^{\left(  n\right)
},r\right)   \; \forall r\in S,\\%
\left(  u_{S}^{\left(  0\right)  },w\right)  & = \left(  u_{0},w\right)\\
\left(  v_{S}^{\left(  1/2\right)  },w\right)  &=\left(  v_{0},w\right)  +\frac{\Delta t}{2}\left(  F^{\left(  0\right)
}\left(  w\right)  -a\left(  u_{0},w\right)  \right) \;\forall w\in S.
\end{split}
\label{eq1}%
\end{equation}

The elimination of $v_{S}^{\left(  n+1/2\right)  }$ from the second equation
by using the first one leads to the operator equation%
\begin{subequations}
\label{deffrakS}
\end{subequations}%
\begin{equation}
\left(
\begin{array}
[c]{c}%
v_{S}^{\left(  n+1/2\right)  }\\
u_{S}^{\left(  n+1\right)  }%
\end{array}
\right)  =\mathfrak{S}\left(
\begin{array}
[c]{c}%
v_{S}^{\left(  n-1/2\right)  }\\
u_{S}^{\left(  n\right)  }%
\end{array}
\right)  +\left(  \Delta t\right)  f_{S}^{\left(  n\right)  }\left(
\begin{array}
[c]{c}%
1\\
\Delta t
\end{array}
\right)  \tag{%
\ref{deffrakS}%
a}\label{deffrakSa}%
\end{equation}
with $A_{S,p}$ as in (\ref{defASp}), }$f_{S}^{\left(  n\right)  }${ as in
(\ref{deffnsn}), and
\begin{equation}
\mathfrak{S:}=\left[
\begin{array}
[c]{rr}%
I_{S} & -\Delta tA_{S,p}\\
\Delta tI_{S} & I_{S}-\Delta t^{2}A_{S,p}%
\end{array}
\right]  .\tag{%
\ref{deffrakS}%
b}\label{deffrakSb}%
\end{equation}
}

Next, we will derive a recursion for the error
\[
e_{v}^{\left(  n+1/2\right)  }=v\left(  t_{n+1/2}\right)  -v_{S}^{\left(
n+1/2\right)  }\quad\text{and\quad}e_{u}^{\left(  n+1\right)  }=u\left(
t_{n+1}\right)  -u_{S}^{\left(  n+1\right)  },
\]
where $u$ is the solution of \eqref{waveeq}-\eqref{waveeqic} and $v$ the
solution of the corresponding first-order formulation: Find $u,v: [0,T]
\rightarrow V$ such that%

\begin{equation}%
\begin{split}
\left(  \dot{v},w\right)  +a\left(  u,w\right)   &  =F\left(  w\right)
\quad\forall\,w\in V,\quad t>0,\\
\left(  v,w\right)   &  =\left(  \dot{u},w\right)  \quad\forall\,w\in V,\quad
t>0,\\
&
\end{split}
\label{waveeq 1st order}%
\end{equation}
and initial conditions $u(0)=u_{0}$ and $v(0)=v_{0}.$

To split the error we introduce the first-order formulation of the
semi-discrete problem (\ref{spacedisc}). Find $u_{S},v_{S}:\left[  0,T\right]
\rightarrow S$ such that%
\[%
\begin{array}
[c]{rl}%
\left.
\begin{array}
[c]{r}%
\left(  \dot{v}_{S},w\right)  +a\left(  u_{S},w\right)  =F\left(  w\right) \\
\left(  v_{S},w\right)  =\left(  \dot{u}_{S},w\right)
\end{array}
\right\}  & \forall w\in S,\quad t>0,\\
\left.
\begin{array}
[c]{c}%
\left(  u_{S}\left(  0\right)  ,w\right)  =\left(  u_{0},w\right) \\
\\
\left(  v_{S}\left(  0\right)  ,w\right)  =\left(  v_{0},w\right)
\end{array}
\right\}  & \forall w\in S.
\end{array}
\]
Hence, we may write $\mathbf{e}^{\left(  n+1\right)  }:=\left(  e_{v}^{\left(
n+\frac{1}{2}\right)  },e_{u}^{\left(  n+1\right)  }\right)  ^{\intercal
}=\mathbf{e}_{S}^{\left(  n+1\right)  }+\mathbf{e}_{S,\Delta t}^{\left(
n+1\right)  }$ with%
\begin{align}
\mathbf{e}_{S}^{\left(  n+1\right)  }  &  :=\left(
\begin{array}
[c]{c}%
e_{v,S}^{\left(  n+1/2\right)  }\\
e_{u,S}^{\left(  n+1\right)  }%
\end{array}
\right)  :=\left(
\begin{array}
[c]{c}%
v\left(  t_{n+1/2}\right)  -v_{S}\left(  t_{n+1/2}\right) \\
u\left(  t_{n+1}\right)  -u_{S}\left(  t_{n+1}\right)
\end{array}
\right)  ,\label{defspliterror}\\
\mathbf{e}_{S,\Delta t}^{\left(  n+1\right)  }  &  :=\left(
\begin{array}
[c]{c}%
e_{v,S,\Delta t}^{\left(  n+1/2\right)  }\\
e_{u,S,\Delta t}^{\left(  n+1\right)  }%
\end{array}
\right)  :=\left(
\begin{array}
[c]{c}%
v_{S}\left(  t_{n+1/2}\right)  -v_{S}^{\left(  n+1/2\right)  }\\
u_{S}\left(  t_{n+1}\right)  -u_{S}^{\left(  n+1\right)  }%
\end{array}
\right)  .
\end{align}

We first investigate the error $\mathbf{e}_{S,\Delta t}^{\left(  n+1\right)  }$
and introduce%
\begin{subequations}
\label{capdelta}
\end{subequations}%
\begin{align}
\Delta_{1}^{\left(  n+1/2\right)  }  &  :=\frac{v_{S}\left(  t_{n+1/2}\right)
-v_{S}\left(  t_{n-1/2}\right)  }{\Delta t}+A_{S,p}u_{S}\left(  t_{n}\right)
-f_{S}^{\left(  n\right)  },\tag{%
\ref{capdelta}%
a}\label{capdeltaa}\\
\Delta_{2}^{\left(  n+1\right)  }  &  :=\frac{u_{S}\left(  t_{n+1}\right)
-u_{S}\left(  t_{n}\right)  }{\Delta t}-v_{S}\left(  t_{n+1/2}\right)  . \tag{%
\ref{capdelta}%
b}\label{capdeltab}%
\end{align}
These equations can be written in the form%
\begin{align}
v_{S}\left(  t_{n+1/2}\right)   &  =v_{S}\left(  t_{n-1/2}\right)  +\left(
\Delta t\right)  \Delta_{1}^{\left(  n+1/2\right)  }-\left(  \Delta t\right)
A_{S,p}u_{S}\left(  t_{n}\right)  +\left(  \Delta t\right)  f_{S}^{\left(
n\right)  },\label{eq3}\\
u_{S}\left(  t_{n+1}\right)   &  =u_{S}\left(  t_{n}\right)  +\left(  \Delta
t\right)  v_{S}\left(  t_{n+1/2}\right)  +\left(  \Delta t\right)  \Delta
_{2}^{\left(  n+1\right)  }. \label{eq4}%
\end{align}
By subtracting the first equation in (\ref{eq1}) from (\ref{eq3}) and the
second equation in (\ref{eq1}) from (\ref{eq4}) we obtain%
\[%
\begin{array}
[c]{rl}%
e_{v,S,\Delta t}^{\left(  n+1/2\right)  }= & e_{v,S,\Delta t}^{\left(
n-1/2\right)  }-\left(  \Delta t\right)  A_{S,p}e_{u,S,\Delta t}^{\left(
n\right)  }+\left(  \Delta t\right)  \Delta_{1}^{\left(  n+1/2\right)  },\\
e_{u,S,\Delta t}^{\left(  n+1\right)  }= & e_{u,S,\Delta t}^{\left(  n\right)
}+\left(  \Delta t\right)  e_{v,S,\Delta t}^{\left(  n+1/2\right)  }+\left(
\Delta t\right)  \Delta_{2}^{\left(  n+1\right)  }.
\end{array}
\]
Eliminating the term $e_{v,S,\Delta t}^{\left(  n+1/2\right)  }$ in the second
equation by using the first one yields%
\[%
\begin{array}
[c]{rl}%
e_{v,S,\Delta t}^{\left(  n+1/2\right)  }= & e_{v,S,\Delta t}^{\left(
n-1/2\right)  }-\left(  \Delta t\right)  A_{S,p}e_{u,S,\Delta t}^{\left(
n\right)  }+\left(  \Delta t\right)  \Delta_{1}^{\left(  n+1/2\right)  },\\
e_{u,S,\Delta t}^{\left(  n+1\right)  }= & \left(  \Delta t\right)
e_{v,S,\Delta t}^{\left(  n-1/2\right)  }+e_{u,S,\Delta t}^{\left(  n\right)
}-\left(  \Delta t\right)  ^{2}A_{S,p}e_{u,S,\Delta t}^{\left(  n\right)  },\\
& +\left(  \Delta t\right)  ^{2}\Delta_{1}^{\left(  n+1/2\right)  }+\left(
\Delta t\right)  \Delta_{2}^{\left(  n+1\right)  }.
\end{array}
\]

We rewrite it in operator form by using the operator $\mathfrak{S}$ as in
(\ref{deffrakS})%
\[
\left(
\begin{array}
[c]{c}%
e_{v,S,\Delta t}^{\left(  n+1/2\right)  }\\
e_{u,S,\Delta t}^{\left(  n+1\right)  }%
\end{array}
\right)  =\mathfrak{S}\left(
\begin{array}
[c]{c}%
e_{v,S,\Delta t}^{\left(  n-1/2\right)  }\\
e_{u,S,\Delta t}^{\left(  n\right)  }%
\end{array}
\right)  +\Delta t\mathfrak{S}_{1}\left(
\begin{array}
[c]{c}%
\Delta_{1}^{\left(  n+1/2\right)  }\\
\Delta_{2}^{\left(  n+1\right)  }%
\end{array}
\right)
\]
with

\[
\mathfrak{S}_{1}=\left[
\begin{array}
[c]{ll}%
I_{S} & 0\\
\left(  \Delta t\right)  I_{S} & I_{S}%
\end{array}
\right]
\]
This recursion can be resolved%
\[
\left(
\begin{array}
[c]{c}%
e_{v,S,\Delta t}^{\left(  n+1/2\right)  }\\
e_{u,S,\Delta t}^{\left(  n+1\right)  }%
\end{array}
\right)  =\mathfrak{S}^{n}\left(
\begin{array}
[c]{c}%
e_{v,S,\Delta t}^{\left(  1/2\right)  }\\
e_{u,S,\Delta t}^{\left(  1\right)  }%
\end{array}
\right)  +\Delta t\sum_{\ell=0}^{n-1}\mathfrak{S}^{\ell}\mathfrak{S}%
_{1}\left(
\begin{array}
[c]{c}%
\Delta_{1}^{\left(  n-\ell+1/2\right)  }\\
\Delta_{2}^{\left(  n+1-\ell\right)  }%
\end{array}
\right)  .\label{eqrecursion}%
\]

Let $I_{S}^{2\times2}:=\left[
\begin{array}
[c]{cc}%
I_{S} & 0\\
0 & I_{S}%
\end{array}
\right]  $ and observe that
\[
\left(  I_{S}^{2\times2}-\mathfrak{S}\right)  ^{-1}=\frac{1}{\Delta t}\left[
\begin{array}
[c]{rr}%
\left(  \Delta t\right)  I_{S} & -I_{S}\\
A_{S,p}^{-1} & 0
\end{array}
\right] \]
and

\[ \left(  I_{S}^{2\times2}-\mathfrak{S}\right)
^{-1}\mathfrak{S}_{1}=\frac{1}{\Delta t}\left[
\begin{array}
[c]{cc}%
0 & -I_{S}\\
A_{S,p}^{-1} & 0
\end{array}
\right]  .\label{eqtwobytwo}%
\]
We introduce%
\begin{align}%
\mbox{\boldmath$ \sigma$}%
^{\left(  n\right)  }  &  =\left(  I_{S}^{2\times2}-\mathfrak{S}\right)
^{-1}\mathfrak{S}_{1}\left(
\begin{array}
[c]{c}%
\Delta_{1}^{\left(  n+1/2\right)  }\\
\Delta_{2}^{\left(  n+1\right)  }%
\end{array}
\right)  =\frac{1}{\Delta t}\left(
\begin{array}
[c]{c}%
-\Delta_{2}^{\left(  n+1\right)  }\\
A_{S,p}^{-1}\Delta_{1}^{\left(  n+1/2\right)  }%
\end{array}
\right) \label{defboldsigmab}\\
&  \overset{\text{(\ref{capdelta})}}{=}\frac{1}{\Delta t}\left(
\begin{array}
[c]{c}%
-\frac{u_{S}\left(  t_{n+1}\right)  -u_{S}\left(  t_{n}\right)  }{\Delta
t}+v_{S}\left(  t_{n+1/2}\right) \\
u_{S}\left(  t_{n}\right)  +A_{S,p}^{-1}\left(  \frac{v_{S}\left(
t_{n+1/2}\right)  -v_{S}\left(  t_{n-1/2}\right)  }{\Delta t}-f_{S}^{\left(
n\right)  }\right)
\end{array}
\right) \nonumber
\end{align}
and the differences%
\begin{align*}
\operatorname*{diff}\nolimits^{\left(  n\right)  }  &  :=\left(
\begin{array}
[c]{l}%
\operatorname*{diff}\nolimits_{1}^{\left(  n-1/2\right)  }\\
\operatorname*{diff}\nolimits_{2}^{\left(  n\right)  }%
\end{array}
\right)  :=%
\mbox{\boldmath$ \sigma$}%
^{\left(  n\right)  }-%
\mbox{\boldmath$ \sigma$}%
^{\left(  n+1\right)  }\\
&  =\!\!\left(\!\!\!
\begin{array}
[c]{c}%
\frac{u_{S}\left(  t_{n+2}\right)  -2u_{S}\left(  t_{n+1}\right)
+u_{S}\left(  t_{n}\right)  }{\Delta t^{2}}+\frac{v_{S}\left(  t_{n+1/2}%
\right)  -v_{S}\left(  t_{n+3/2}\right)  }{\Delta t}\\
\frac{u_{S}\left(  t_{n}\right)  -u_{S}\left(  t_{n+1}\right)  }{\Delta
t}+A_{S,p}^{-1}\!\left(  \frac{-v_{S}\left(  t_{n+3/2}\right)  +2v_{S}\left(
t_{n+1/2}\right)  -v_{S}\left(  t_{n-1/2}\right)  }{\Delta t^{2}}+\frac
{f_{S}^{\left(  n+1\right)  }-f_{S}^{\left(  n\right)  }}{\Delta t}\!\right)
\end{array}
\!\!\!\right)
\end{align*}
and use (\ref{eqtwobytwo}) to rewrite the error representation
(\ref{eqrecursion}) as
\begin{align}
\left(
\begin{array}
[c]{c}%
e_{v,S,\Delta t}^{\left(  n+1/2\right)  }\\
e_{u,S,\Delta t}^{\left(  n+1\right)  }%
\end{array}
\right)   &  =\mathfrak{S}^{n}\left(
\begin{array}
[c]{c}%
e_{v,S,\Delta t}^{\left(  1/2\right)  }\\
e_{u,S,\Delta t}^{\left(  1\right)  }%
\end{array}
\right)  +\Delta t\sum_{\ell=0}^{n-1}\mathfrak{S}^{\ell}\left(  I_{S}%
^{2\times2}-\mathfrak{S}\right)
\mbox{\boldmath$ \sigma$}%
^{\left(  n-\ell\right)  }\nonumber\\
&  =\mathfrak{S}^{n}\left(
\begin{array}
[c]{c}%
e_{v,S,\Delta t}^{\left(  1/2\right)  }\\
e_{u,S,\Delta t}^{\left(  1\right)  }%
\end{array}
\right)  +\Delta t\sum_{\ell=1}^{n-1}\mathfrak{S}^{\ell}\operatorname*{diff}%
\nolimits^{\left(  n-\ell\right)  }\nonumber\\
&  +\Delta t%
\mbox{\boldmath$ \sigma$}%
^{\left(  n\right)  }-\Delta t\mathfrak{S}^{n}%
\mbox{\boldmath$ \sigma$}%
^{\left(  1\right)  }. \label{errorreprfu}%
\end{align}

\subsubsection{Stability}

As usual, the convergence analysis can be split
into an estimate for the stability of the iteration operator $\mathfrak{S}$
(corresponding to a homogeneous right-hand side) and a consistency estimate.
We begin with the analysis of the stability.

\begin{theorem}
[Stability]\label{TheoStabLeapFrog}Let the CFL condition (\ref{CFL}) be
satisfied. Then the leap-frog scheme (\ref{leap_frog_lts_fd}) is stable%
\[
\left\Vert v_{S}^{\left(  n{+1/2}\right)  }\right\Vert +\left\Vert
u_{S}^{\left(  n\right)  }\right\Vert \leq C_{0}\left(  \left\Vert
v_{S}^{\left(  1/2\right)  }\right\Vert +\left\Vert u_{S}^{\left(  1\right)
}\right\Vert \right)  ,
\]
where $C_{0}$ is independent of $n$, $\Delta t$, $h$, and $T$.
\end{theorem}

\begin{proof}
We choose the eigensystem as introduced in Corollary \ref{Coreigsys} and
expand%
\[
u_{S}^{\left(  n\right)  }=\sum_{k=1}^{M}\chi_{S,p,k}^{\left(  n\right)  }%
\eta_{S,p,k}\quad\text{and\quad}v_{S}^{\left(  n-1/2\right)  }=\sum_{k=1}%
^{M}\beta_{S,p,k}^{\left(  n-1/2\right)  }\eta_{S,p,k}.
\]
Inserting this into the recursion $\left(
\begin{array}
[c]{c}%
v_{S}^{\left(  n+1/2\right)  }\\
u_{S}^{\left(  n+1\right)  }%
\end{array}
\right)  =\mathfrak{S}\left(
\begin{array}
[c]{c}%
v_{S}^{\left(  n-1/2\right)  }\\
u_{S}^{\left(  n\right)  }%
\end{array}
\right)  $ leads to a recursion for the coefficients $\beta_{S,p,k}^{\left(
n+1/2\right)  }$, $\chi_{S,p,k}^{\left(  n+1\right)  }$:%
\begin{equation}
\left(
\begin{array}
[c]{c}%
\beta_{S,p,k}^{\left(  n+1/2\right)  }\\
\chi_{S,p,k}^{\left(  n+1\right)  }%
\end{array}
\right)  =\mathbf{S}_{p}\left(
\begin{array}
[c]{c}%
\beta_{S,p,k}^{\left(  n-1/2\right)  }\\
\chi_{S,p,k}^{\left(  n\right)  }%
\end{array}
\right)  \label{recursionu}%
\end{equation}
with%
\[
\mathbf{S}_{p}=\left(
\begin{array}
[c]{rr}%
1 & -\left(  \Delta t\right)  \lambda_{S,p,k}\\
\Delta t & 1-\left(  \Delta t\right)  ^{2}\lambda_{S,p,k}%
\end{array}
\right)  .
\]
The eigenvalues of $\mathbf{S}_{p}$ are given by%
\[
1-\frac{\lambda_{S,p,k}\left(  \Delta t\right)  ^{2}}{2}\pm\frac
{\operatorname*{i}\Delta t}{2}\sqrt{\lambda_{S,p,k}\left(  4-\lambda
_{S,p,k}\left(  \Delta t\right)  ^{2}\right)  }.
\]
The CFL condition (\ref{CFL}) implies $\left(  \Delta t\right)  ^{2}%
\lambda_{p}^{\max}<4$ so that the eigenvalues are different and $\mathbf{S}%
_{p}$ is diagonalizable. From \cite[Satz (6.9.2)(2)]{StoerII} we conclude that
there is a norm $\left\vert
\kern-.1em%
\left\vert
\kern-.1em%
\left\vert \cdot\right\vert
\kern-.1em%
\right\vert
\kern-.1em%
\right\vert $ in $\mathbb{R}^{2}$ such that the associated matrix norm
$\left\vert
\kern-.1em%
\left\vert
\kern-.1em%
\left\vert \mathbf{S}_{p}\right\vert
\kern-.1em%
\right\vert
\kern-.1em%
\right\vert $ is bounded from above by the spectral radius:%
\[
\rho\left(  \mathbf{S}_{p}\right)  =\max_{\pm}\left\vert 1-\frac
{\lambda_{S,p,k}\left(  \Delta t\right)  ^{2}}{2}\pm\frac{\operatorname*{i}%
\Delta t}{2}\sqrt{\lambda_{S,p,k}\left(  4-\lambda_{S,p,k}\left(  \Delta
t\right)  ^{2}\right)  }\right\vert =1\text{.}%
\]
Hence%
\[
\left\vert
\kern-.1em%
\left\vert
\kern-.1em%
\left\vert \left(
\begin{array}
[c]{c}%
\beta_{S,p,k}^{\left(  n+1/2\right)  }\\
\chi_{S,p,k}^{\left(  n+1\right)  }%
\end{array}
\right)  \right\vert
\kern-.1em%
\right\vert
\kern-.1em%
\right\vert \leq\left\vert
\kern-.1em%
\left\vert
\kern-.1em%
\left\vert \left(
\begin{array}
[c]{c}%
\beta_{S,p,k}^{\left(  1/2\right)  }\\
\chi_{S,p,k}^{\left(  1\right)  }%
\end{array}
\right)  \right\vert
\kern-.1em%
\right\vert
\kern-.1em%
\right\vert .
\]
Since all norms in $\mathbb{R}^{2}$ are equivalent there exists a constant $C$
such that%
\begin{equation}
\sqrt{\left\vert \chi_{S,p,k}^{\left(  n\right)  }\right\vert ^{2}+\left\vert
\beta_{S,p,k}^{\left(  n-1/2\right)  }\right\vert ^{2}}\leq C\sqrt{\left\vert
\beta_{S,p,k}^{\left(  1/2\right)  }\right\vert ^{2}+\left\vert \chi
_{S,p,k}^{\left(  1\right)  }\right\vert ^{2}}. \label{alphaest}%
\end{equation}
The eigenfunctions $\eta_{S,p,k}$ are chosen to be an orthonormal system in
$L^{2}\left(  \Omega\right)  $ so that%
\begin{align}
\left\Vert v_{S}^{\left(  n+1/2\right)  }\right\Vert ^{2}+\left\Vert
u_{S}^{\left(  n\right)  }\right\Vert ^{2}  &  =\sum_{k=1}^{M}\left\vert
\chi_{S,p,k}^{\left(  n\right)  }\right\vert ^{2}+\left\vert \beta
_{S,p,k}^{\left(  n+1/2\right)  }\right\vert ^{2}\leq C^{2}\sum_{k=1}%
^{M}\left(  \left\vert \beta_{S,p,k}^{\left(  1/2\right)  }\right\vert
^{2}+\left\vert \chi_{S,p,k}^{\left(  1\right)  }\right\vert ^{2}\right)
\label{usn}\\
&  =C^{2}\left(  \left\Vert v_{S}^{\left(  1/2\right)  }\right\Vert
^{2}+\left\Vert u_{S}^{\left(  1\right)  }\right\Vert ^{2}\right) \nonumber
\end{align}
which shows the $L^{2}\left(  \Omega\right)  $-stability of the method.
\end{proof}

\subsubsection{Error Estimates}

In this section we first estimate the discrete error $e_{u,S,\Delta t}^{\left(
n+1\right)  }$. Standard estimates on the semi-discrete error then lead to an estimate of the total error $e_u^{(n+1)}$.

\begin{theorem}
\label{Theotimedisc}Let the assumptions of Lemma \ref{LemapH1est} be
satisfied. Let the solution of the semi-discrete equation (\ref{spacedisc})
satisfy $u_{S}\in W^{5,\infty}\left(  \left[  0,T\right]  ;L^{2}\left(
\Omega\right)  \right)  $ and the right-hand side $f_{S}\in W^{3,\infty
}\left(  \left[  0,T\right]  ;L^{2}\left(  \Omega\right)  \right)  $. Then the
fully discrete solution $u_{S}^{\left(  n+1\right)  }$ of
(\ref{leap_frog_lts_fd}) satisfies the error estimate%
\[
\left\Vert e_{u,S,\Delta t}^{\left(  n+1\right)  }\right\Vert \leq C\Delta
t^{2}\left(  1+T\right)  \mathcal{M}\left(  u_{S},f_{S}\right)
\]
with
\begin{equation}
\mathcal{M}\left(  u_{S},f_{S}\right)  :=\max\left\{  \max_{1\leq\ell\leq
3}\left\Vert \partial_{t}^{\ell}f_{S}\right\Vert _{L^{\infty}\left(  \left[
0,T\right]  ;L^{2}\left(  \Omega\right)  \right)  },\max_{3\leq\ell\leq
5}\left\Vert \partial_{t}^{\ell}u_{S}\right\Vert _{L^{\infty}\left(  \left[
0,T\right]  ;L^{2}\left(  \Omega\right)  \right)  }\right\}  \label{defcalM}%
\end{equation}
and a constant $C$ which is independent of $n$, $\Delta t$, $T$, $h$, $p$, $f_{S}$,
and $u_{S}$.
\end{theorem}

\begin{proof}
We apply the stability estimate to the second component of the error
representation (\ref{errorreprfu}). From Theorem \ref{TheoStabLeapFrog} and
(\ref{defboldsigmab}) we obtain\footnote{For a pair of functions
$\mathbf{v}=\left(  v_{1},v_{2}\right)  ^{\intercal}\in S^{2}$ we use the
notation $\left\Vert \mathbf{v}\right\Vert _{\ell^{1}}:=\left\Vert
v_{1}\right\Vert +\left\Vert v_{2}\right\Vert $.}%
\begin{align}
\left\Vert e_{u,S,\Delta t}^{\left(  n+1\right)  }\right\Vert  &  \leq
C_{0}\left\Vert \mathbf{e}_{S,\Delta t}^{\left(  1\right)  }\right\Vert
_{\ell^{1}}+C_{0}\Delta t\sum_{\ell=1}^{n-1}\left\Vert \operatorname*{diff}%
\nolimits^{\left(  n-\ell\right)  }\right\Vert _{\ell^{1}}
\label{eestmainform}\\
&  +\Delta t\left\Vert
\mbox{\boldmath$ \sigma$}%
^{\left(  n\right)  }\right\Vert _{\ell^{1}}+C_{0}\Delta t\left\Vert
\mbox{\boldmath$ \sigma$}%
^{\left(  1\right)  }\right\Vert _{\ell^{1}}.\nonumber
\end{align}
For the summands in the second term of the right-hand side in
(\ref{eestmainform}), we obtain by a Taylor argument and Corollary
\ref{CorAspL2}%
\begin{equation}
\operatorname*{diff}\nolimits^{\left(  n\right)  }=\left(
\begin{array}
[c]{c}%
0\\
-\dot{u}_{S}\left(  t_{n+1/2}\right)  +A_{S,p}^{-1}\left(  -\ddot{v}%
_{S}\left(  t_{n+1/2}\right)  +\dot{f}_{S}\left(  t_{n+1/2}\right)  \right)
\end{array}
\right)  +\frac{\left(  \Delta t\right)  ^{2}}{24}\mathcal{E}_{n}%
^{\operatorname{I}} \label{titb}%
\end{equation}
with%
\[
\left\Vert \mathcal{E}_{n}^{\operatorname{I}}\right\Vert _{\ell^{1}}%
\leq2\left(  1+\frac{3}{c_{\operatorname*{coer}}}\right)  \mathcal{M}%
_{n}\left(  u_{S},f_{S}\right)
\]
and%
\[
\mathcal{M}_{n}\left(  u_{S},f_{S}\right)  :=\max\left\{  \max_{1\leq\ell
\leq3}\left\Vert \partial_{t}^{\ell}f_{S}\right\Vert _{L^{\infty}\left(
\left[  t_{n},t_{n+1}\right]  ;L^{2}\left(  \Omega\right)  \right)  }%
,\max_{3\leq\ell\leq5}\left\Vert \partial_{t}^{\ell}u_{S}\right\Vert
_{L^{\infty}\left(  \left[  t_{n-1/2},t_{n+2}\right]  ;L^{2}\left(
\Omega\right)  \right)  }\right\}  .
\]

Now, let $\psi$ denote the second component of the first term in the right-hand side of
(\ref{titb}),
\[
\psi   :=-\dot{u}_{S}\left(  t_{n+1/2}\right)  +A_{S,p}^{-1}\left(  -\ddot
{v}_{S}\left(  t_{n+1/2}\right)  +\dot{f}_{S}\left(  t_{n+1/2}\right)
\right) .
\]
By using $\ddot{u}_{S}+A_{S}u_{S}=f_{S}$  (cf. (\ref{spacedisca}) and
(\ref{defASp})) we obtain%
\begin{align*}
\psi &
 =-\partial_{t}\left(  u_{S}\left(  t_{n+1/2}\right)  -A_{S,p}%
^{-1}A_{S}u_{S}\left(  t_{n+1/2}\right)  \right)  \\
&  =\frac{2}{p^{2}}A_{S,p}^{-1}\sum_{j=1}^{p-1}\alpha_{j}^{p}\left(
\frac{\Delta t}{p}\right)  ^{2j}\left(  A_{S}R_{\mathcal{N}}\right)  ^{j}%
A_{S}\dot{u}_{S}\left(  t_{n+1/2}\right)  \\
&  =\left(  I_{S}-\frac{2}{p^{2}}\sum_{j=1}^{p-1}\alpha_{j}^{p}\left(
\frac{\Delta t}{p}\right)  ^{2j}\left(  R_{\mathcal{N}}A_{S}\right)
^{j}\right)  ^{-1}\frac{2\left(  \Delta t\right)  ^{2}}{p^{4}}R_{\mathcal{N}%
}\sum_{j=1}^{p-1}\alpha_{j}^{p}\left(  \frac{\Delta t}{p}\right)  ^{2\left(
j-1\right)  }\left(  A_{S}R_{\mathcal{N}}\right)  ^{j-1}A_{S}\dot{u}%
_{S}\left(  t_{n+1/2}\right)  .
\end{align*}
{We employ (\ref{L2Ap}) and argue as in the proof of }Corollary \ref{CorAspL2}
to obtain%
\begin{align*}
\left\Vert \psi\right\Vert  &  \leq2\left\Vert R_{\mathcal{N}}^{1/2}\frac
{2}{p^{2}}\sum_{j=1}^{p-1}\alpha_{j}^{p}\left(  \frac{\Delta t}{p}\right)
^{2\left(  j-1\right)  }\left(  R_{\mathcal{N}}^{1/2}A_{S}R_{\mathcal{N}%
}^{1/2}\right)  ^{j-1}\left(  \frac{\Delta t}{p}\right)  ^{2}R_{\mathcal{N}%
}^{1/2}A_{S}\,\dot{u}_{S}\left(  t_{n+1/2}\right)  \right\Vert \\
&  \leq 2 \frac{\left(  \Delta t\right)  ^{2}}{12\,c_{\operatorname*{eq}}^{2}%
}\left\Vert A_{S}\,\dot{u}_{S}\left(  t_{n+1/2}\right)  \right\Vert.
\end{align*}
This yields%
\begin{align*}
\left\Vert -\dot{u}_{S}\left(  t_{n+1/2}\right)  +A_{S,p}^{-1}\left(
-\ddot{v}_{S}\left(  t_{n+1/2}\right)  +\dot{f}_{S}\left(  t_{n+1/2}\right)
\right)  \right\Vert  &  \leq\frac{\left(  \Delta t\right)  ^{2}%
}{6c_{\operatorname*{eq}}^{2}}\left\Vert A_{S}\dot{u}_{S}\left(
t_{n+1/2}\right)  \right\Vert \\
&  \leq\frac{\left(  \Delta t\right)  ^{2}}{6c_{\operatorname*{eq}}^{2}%
}\left(  \left\Vert \partial_{t}^{3}u_{S}\left(  t_{n+1/2}\right)  \right\Vert
+\left\Vert \dot{f}_{S}^{\left(  n+1/2\right)  }\right\Vert \right)  .
\end{align*}

In summary we have proved%
\[
\left\Vert \operatorname*{diff}\nolimits^{\left(  n\right)  }\right\Vert
_{\ell^{1}}\leq\frac{\left(  \Delta t\right)  ^{2}}{12}\left(  1+\frac
{8}{c_{\operatorname*{eq}}^{2}}+\frac{3}{c_{\operatorname*{coer}}}\right)
\mathcal{M}_{n}\left(  u_{S},f_{S}\right)  .
\]

Next, we estimate the remaining terms in (\ref{eestmainform}). We employ the
discrete wave equation and a Taylor argument to obtain%
\begin{align}
\Delta t\left\Vert
\mbox{\boldmath$ \sigma$}%
^{\left(  n\right)  }\right\Vert _{\ell^{1}}  &  \leq\frac{\left(  \Delta
t\right)  ^{2}}{24}\left\Vert \partial_{t}^{3}u_{S}\right\Vert _{L^{\infty
}\left(  \left[  t_{n},t_{n+1}\right]  ;L^{2}\left(  \Omega\right)  \right)
}\label{boldsigman}\\
&  +\left\Vert A_{S,p}^{-1}\left(  \underset{=0}{\underbrace{A_{S,p}%
u_{S}\left(  t_{n}\right)  +\ddot{u}_{S}\left(  t_{n}\right)  -f_{S}^{\left(
n\right)  }}}+\frac{\dot{u}_{S}\left(  t_{n+1/2}\right)  -\dot{u}_{S}\left(
t_{n-1/2}\right)  }{\Delta t}-\ddot{u}_{S}\left(  t_{n}\right)  \right)
\right\Vert \\
&  \overset{\text{Cor. \ref{CorAspL2}}}{\leq}\frac{\left(  \Delta t\right)
^{2}}{24}\left\Vert \partial_{t}^{3}u_{S}\right\Vert _{L^{\infty}\left(
\left[  t_{n},t_{n+1}\right]  ;L^{2}\left(  \Omega\right)  \right)
}\nonumber\\
&  +\frac{2}{c_{\operatorname*{coer}}}\left\Vert \frac{\dot{u}_{S}\left(
t_{n+1/2}\right)  -\dot{u}_{S}\left(  t_{n-1/2}\right)  }{\Delta t}-\ddot
{u}_{S}\left(  t_{n}\right)  \right\Vert \\
&  \leq\frac{\left(  \Delta t\right)  ^{2}}{24}\left\Vert \partial_{t}%
^{3}u_{S}\right\Vert _{L^{\infty}\left(  \left[  t_{n},t_{n+1}\right]
;L^{2}\left(  \Omega\right)  \right)  }+\frac{2}{c_{\operatorname*{coer}}%
}\frac{(\Delta t)^{2}}{24}\left\Vert \partial_{t}^{4}u_{S}\right\Vert
_{L^{\infty}\left(  \left[  t_{n},t_{n+1}\right]  ;L^{2}\left(  \Omega\right)
\right)  }\nonumber\\
&  \leq\frac{\left(  \Delta t\right)  ^{2}}{24}\left(  1+\frac{2}%
{c_{\operatorname*{coer}}}\right)  \mathcal{M}_{n}\left(  u_{S},f_{S}\right)
.\nonumber
\end{align}
The estimate of the last term in (\ref{eestmainform}) follows by setting $n=1$
in (\ref{boldsigman})%
\[
C_{0}\Delta t\left\Vert
\mbox{\boldmath$ \sigma$}%
^{\left(  1\right)  }\right\Vert _{\ell^{1}}\leq C_{0}\frac{\left(  \Delta
t\right)  ^{2}}{24}\left(  1+\frac{2}{c_{\operatorname*{coer}}}\right)
\mathcal{M}_{1}\left(  u_{S},f_{S}\right)  .
\]
Inserting these estimates into (\ref{eestmainform}) leads to%
\begin{align}
\left\Vert e_{u,S,\Delta t}^{\left(  n+1\right)  }\right\Vert  &  \leq
C_{0}\left\Vert \mathbf{e}_{S,\Delta t}^{\left(  1\right)  }\right\Vert
_{\ell^{1}}+C_{0}\frac{\left(  \Delta t\right)  ^{2}}{12}\left(  1+\frac
{8}{c_{\operatorname*{eq}}^{2}}+\frac{3}{c_{\operatorname*{coer}}}\right)
\Delta t\sum_{\ell=1}^{n-1}\mathcal{M}_{n-\ell}\left(  u_{S},f_{S}\right)
\label{ruspu}\\
&  +\frac{\left(  \Delta t\right)  ^{2}}{24}\left(  1+\frac{2}%
{c_{\operatorname*{coer}}}\right)  \left(  \mathcal{M}_{n}\left(  u_{S}%
,f_{S}\right)  +C_{0}\mathcal{M}_{1}\left(  u_{S},f_{S}\right)  \right) \\
&  \leq C_{0}\left\Vert \mathbf{e}_{S,\Delta t}^{\left(  1\right)
}\right\Vert _{\ell^{1}}+\frac{\left(  \Delta t\right)  ^{2}}{12}\left(
C_{0}T\left(  1+\frac{8}{c_{\operatorname*{eq}}^{2}}+\frac{3}%
{c_{\operatorname*{coer}}}\right)  +\left(  1+\frac{2}{c_{\operatorname*{coer}%
}}\right)  \frac{1+C_{0}}{2}\right)  \mathcal{M}\left(  u_{S},f_{S}\right)
\end{align}

It remains to estimate the initial error $\mathbf{e}_{S,\Delta t}^{\left(
1\right)  }$. Let $u_{S}^{\left(  0\right)  }:=u_{S}\left(  0\right)  $ and
$v_{S}^{\left(  0\right)  }:=\dot{u}_{S}\left(  0\right)  \in S$ be as in
(\ref{spacediscb}). A Taylor argument for some $0\leq\theta\leq\tau\leq\Delta
t$ and the definition of $u_{S}^{\left(  0\right)  }$, $u_{S}^{\left(
1\right)  }$ as in (\ref{leap_frog_lts_fd}) lead to
\begin{align}
\left\Vert u_{S}\left(  t_{1}\right)  -u_{S}^{\left(  1\right)  }\right\Vert
&  \leq\left\Vert \left(  u_{S}^{\left(  0\right)  }+\left(  \Delta t\right)
v_{S}^{\left(  0\right)  }+\frac{\Delta t^{2}}{2}\ddot{u}_{S}\left(
\tau\right)  \right)  -\left(  u_{S}^{\left(  0\right)  }+\left(  \Delta
t\right)  v_{S}^{\left(  0\right)  }+\frac{\Delta t^{2}}{2}\left(
f_{S}^{\left(  0\right)  }-A_{S}u_{S}^{\left(  0\right)  }\right)  \right)
\right\Vert \label{initerr1}\\
&  =\frac{\Delta t^{2}}{2}\left\Vert f_{S}\left(  \tau\right)  -f_{S}^{\left(
0\right)  }-A_{S}\left(  u_{S}\left(  \tau\right)  -u_{S}^{\left(  0\right)
}\right)  \right\Vert \nonumber\\
&  \leq\frac{\Delta t^{3}}{2}\left(  \left\Vert \dot{f}_{S}\right\Vert
_{L^{\infty}\left(  \left[  0,\Delta t\right]  ;L^{2}\left(  \Omega\right)
\right)  }+\left\Vert A_{S}\dot{u}_{S}\left(  \theta\right)  \right\Vert
\right)  \nonumber\\
&  \leq\frac{\Delta t^{3}}{2}\left(  2\left\Vert \dot{f}_{S}\right\Vert
_{L^{\infty}\left(  \left[  0,\Delta t\right]  ;L^{2}\left(  \Omega\right)
\right)  }+\left\Vert \partial_{t}^{3}u_{S}\right\Vert _{L^{\infty}\left(
\left[  0,\Delta t\right]  ;,L^{2}\left(  \Omega\right)  \right)  }\right)
\nonumber\\
&  \leq\frac{3}{2}\Delta t^{3}\mathcal{M}\left(  u_{S},f_{S}\right)
.\nonumber
\end{align}
For the initial error in $v_{S}$ we obtain by a similar Taylor argument%
\begin{align}
\left\Vert v_{S}\left(  t_{1/2}\right)  -v_{S}^{\left(  1/2\right)
}\right\Vert  &  =\left\Vert \dot{u}_{S}\left(  t_{1/2}\right)  -v_{S}%
^{\left(  0\right)  }-\frac{\Delta t}{2}\left(  f_{S}^{\left(  0\right)
}-A_{S}u_{S,0}\right)  \right\Vert \label{initerr2}\\
&  =\frac{\Delta t}{2}\left\Vert \ddot{u}_{S}\left(  \tau\right)  +A_{S}%
u_{S}^{\left(  0\right)  }-f_{S}^{\left(  0\right)  }\right\Vert \nonumber\\
&  =\frac{\Delta t}{2}\left\Vert \ddot{u}_{S}\left(  \tau\right)  +A_{S}%
u_{S}\left(  \tau\right)  -f_{S}\left(  \tau\right)  +A_{S}\left(
u_{S}^{\left(  0\right)  }-u_{S}\left(  \tau\right)  \right)  +f_{S}\left(
\tau\right)  -f_{S}^{\left(  0\right)  }\right\Vert \nonumber\\
&  \leq\frac{\left(  \Delta t\right)  ^{2}}{2}\left(  \left\Vert \partial
_{t}^{3}u_{S}\right\Vert _{L^{\infty}\left(  \left[  0,\Delta t\right]
;L^{2}\left(  \Omega\right)  \right)  }+2\left\Vert \dot{f}_{S}\right\Vert
_{L^{\infty}\left(  \left[  0,\Delta t\right]  ;L^{2}\left(  \Omega\right)
\right)  }\right)  \nonumber\\
&  \leq\frac{3\left(  \Delta t\right)  ^{2}}{2}\mathcal{M}\left(  u_{S}%
,f_{S}\right)  .\nonumber
\end{align}

In summary, we have estimated the initial error by%
\begin{equation}
\left\Vert \mathbf{e}_{S,\Delta t}^{\left(  1\right)  }\right\Vert _{\ell^{1}%
}\leq\frac{3\left(  \Delta t\right)  ^{2}}{2}\left(  1+\Delta t\right)
\mathcal{M}\left(  u_{S},f_{S}\right)  . \label{initerror}%
\end{equation}
The combination of (\ref{ruspu}) and (\ref{initerror}) leads to the assertion.
\end{proof}

Theorem \ref{Theotimedisc} can be combined with known error estimates for the
semi-discrete error $\mathbf{e}_{S}^{\left(  n+1\right)  }$ to obtain an error
estimate of the total error.

\begin{theorem}
Let the bilinear form $a\left(  \cdot,\cdot\right)  $ satisfy (\ref{wellposed}%
) and let the CFL condition (\ref{CFL}) hold. Assume that the exact solution
satisfies
$u\in W^{1,\infty}\left(  \left[  0,T\right]  ;H^{m+1}\left(
\Omega\right)  \right)  \cap W^{5,\infty}\left(  \left[  0,T\right]
;L^{2}\left(  \Omega\right)  \right)  $.
Then, the corresponding fully discrete Galerkin FE formulation
with local time-stepping (\ref{leap_frog_lts_fd}) has a unique
solution $u_{S}^{\left(  n+1\right)  }$ which satisfies the error estimate%
\[
\left\Vert u(t_{n+1}) - u_{S}^{\left(  n+1\right)  }\right\Vert \leq C\left(  1+T\right)
\left(  h^{m+1}+\Delta t^{2}\right)  \mathcal{M}\left(  u,u_{S},f_{S}\right)
\]
with%
\[
\mathcal{M}\left(  u,u_{S},f_{S}\right)  :=\max\left\{  \mathcal{M}\left(
u_{S},f_{S}\right)  ,\left\Vert u\right\Vert _{W^{1,\infty}\left(  \left[
0,T\right]  ;H^{m+1}\left(  \Omega\right)  \right)  }\right\}
\]
and a constant $C$ which is independent of $n$, $\Delta t$, $h$, $p$, $f_{S}$,
$u_{S}$, and the final time $T$.
\end{theorem}

\begin{proof}
The existence of the semi-discrete solution $u_{S}$ follows from \cite[Theorem
3.1]{Baker}, which directly implies the existence of our fully discrete LTS-Galerkin FE solution.

Next, we split the total error \[\mathbf{e}^{\left(  n+1\right)
}=\left(  v\left(  t_{n+1/2}\right)  -v_{S}^{\left(  n+1/2\right)  },u\left(
t_{n+1}\right)  -u_{S}^{\left(  n+1\right)  }\right)  ^{\intercal}\] according
to (\ref{defspliterror}). Following \cite{MuellerSchwab16}, we note that
the semi-discrete solution $u_S$ inherits the same regularity from
$u\in W^{5,\infty}\left(  \left[  0,T\right]  ;L^{2}\left(
\Omega\right)  \right)  $; thus, we can apply Theorem \ref{Theotimedisc}.

To estimate the remaining error from the semi-discretization,
\[\mathbf{e}_{S}^{\left(  n+1\right)  }=\left(  v\left(  t_{n+1/2}\right)  -v_{S}\left(
t_{n+1/2}\right)  ,u\left(  t_{n+1}\right)  -u_{S}\left(  t_{n+1}\right)
\right)  ^{\intercal},\]  we use \cite[Theorem 3.1]{Baker} to obtain
\begin{equation}
\left\Vert u-u_{S}\right\Vert _{L^{\infty}\left(  \left[  0,T\right]
;L^{2}\left(  \Omega\right)  \right)  }\leq Ch^{m+1}\left(  \left\Vert
u\right\Vert _{L^{\infty}\left(  \left[  0,T\right]  ;H^{m+1}\left(
\Omega\right)  \right)  }+\left\Vert \dot{u}\right\Vert _{L^{2}\left(  \left[
0,T\right]  ;H^{m+1}\left(  \Omega\right)  \right)  }\right)
.\label{Bakerest}%
\end{equation}
Inspection of the proof in \cite[Theorem 3.1]{Baker} shows that the
constant in (\ref{Bakerest}) can be estimated by $C\left(  1+\sqrt{T}\right)
$. Using a H\"{o}lder inequality in the second summand of the right-hand side
in (\ref{Bakerest}) thus results in%
\[
\left\Vert \dot{u}\right\Vert _{L^{2}\left(  \left[  0,T\right]
;H^{m+1}\left(  \Omega\right)  \right)  }\leq\sqrt{T}\left\Vert \dot
{u}\right\Vert _{L^{\infty}\left(  \left[  0,T\right]  ;H^{m+1}\left(
\Omega\right)  \right)  },%
\]
from which we conclude that%
\[
\left\Vert u-u_{S}\right\Vert _{L^{\infty}\left(  \left[  0,T\right]
;L^{2}\left(  \Omega\right)  \right)  }\leq C^{\prime}h^{m+1}\left(
1+T\right)  \left\Vert u\right\Vert _{W^{1,\infty}\left(  \left[  0,T\right]
;H^{m+1}\left(  \Omega\right)  \right)  }%
\]
with a constant $C^{\prime}$ which is independent of the final time $T$.
Finally, the triangle inequality leads to the assertion.
\end{proof}

\begin{figure}[th]
\centering
\begin{tabular}
[c]{ccc}%
\includegraphics[width=.29\textwidth]{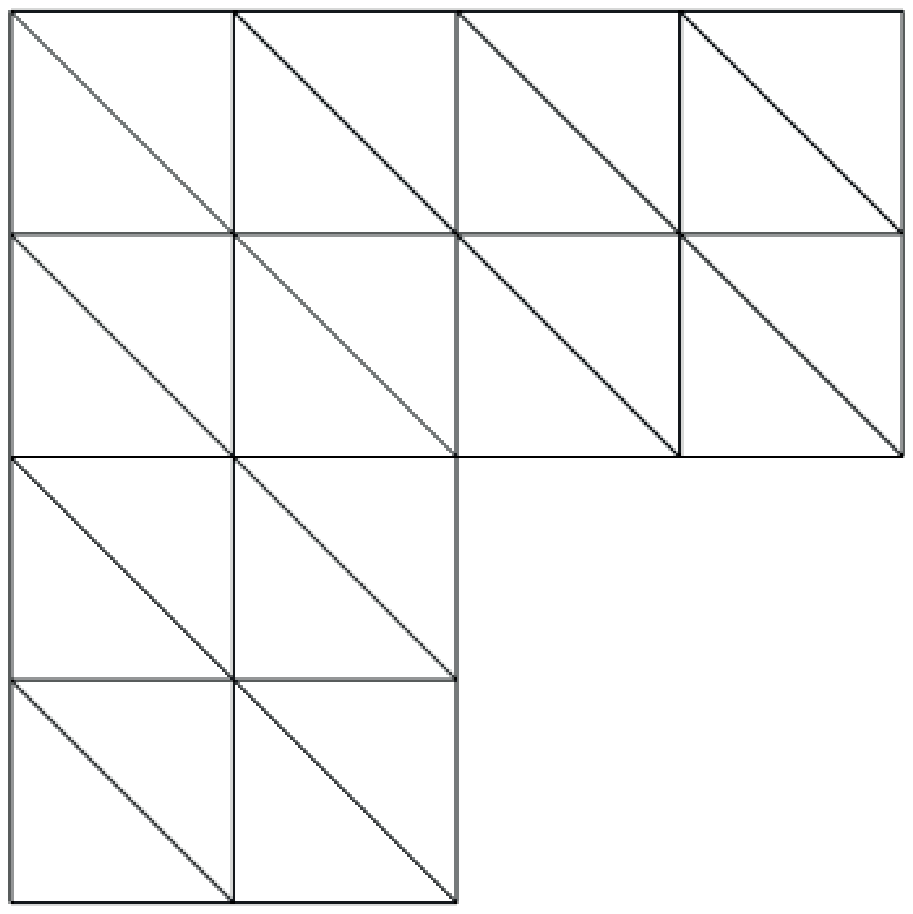} &
\includegraphics[width=.29\textwidth]{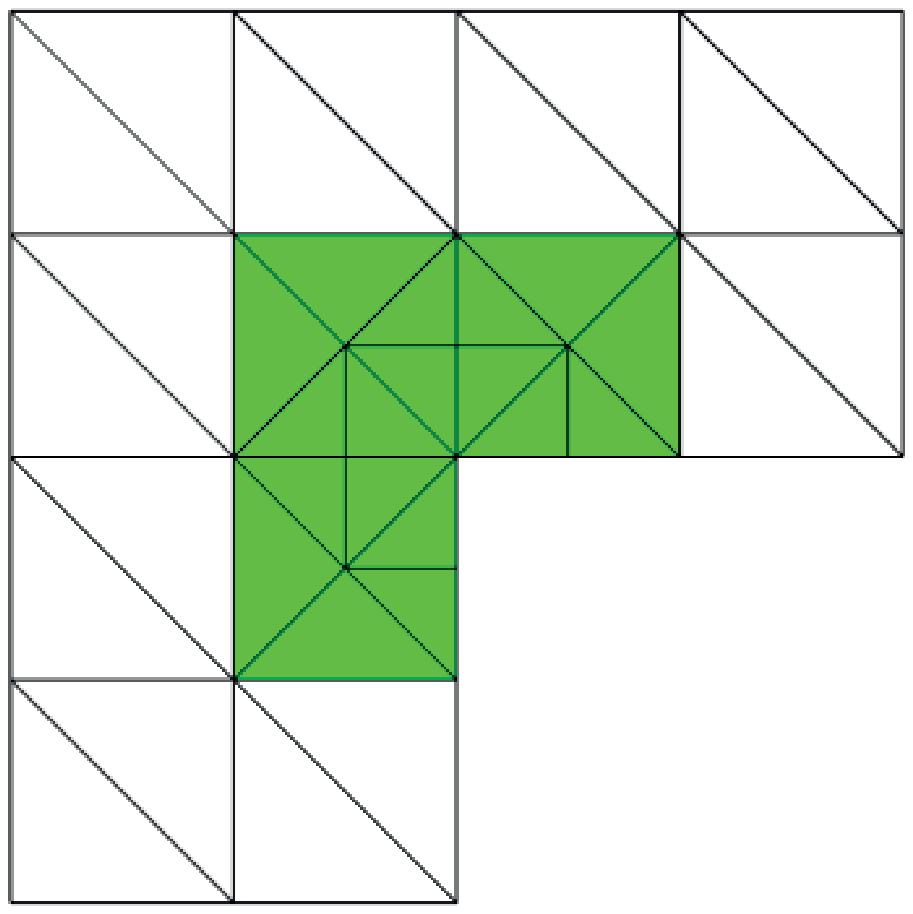} &
\includegraphics[width=.29\textwidth]{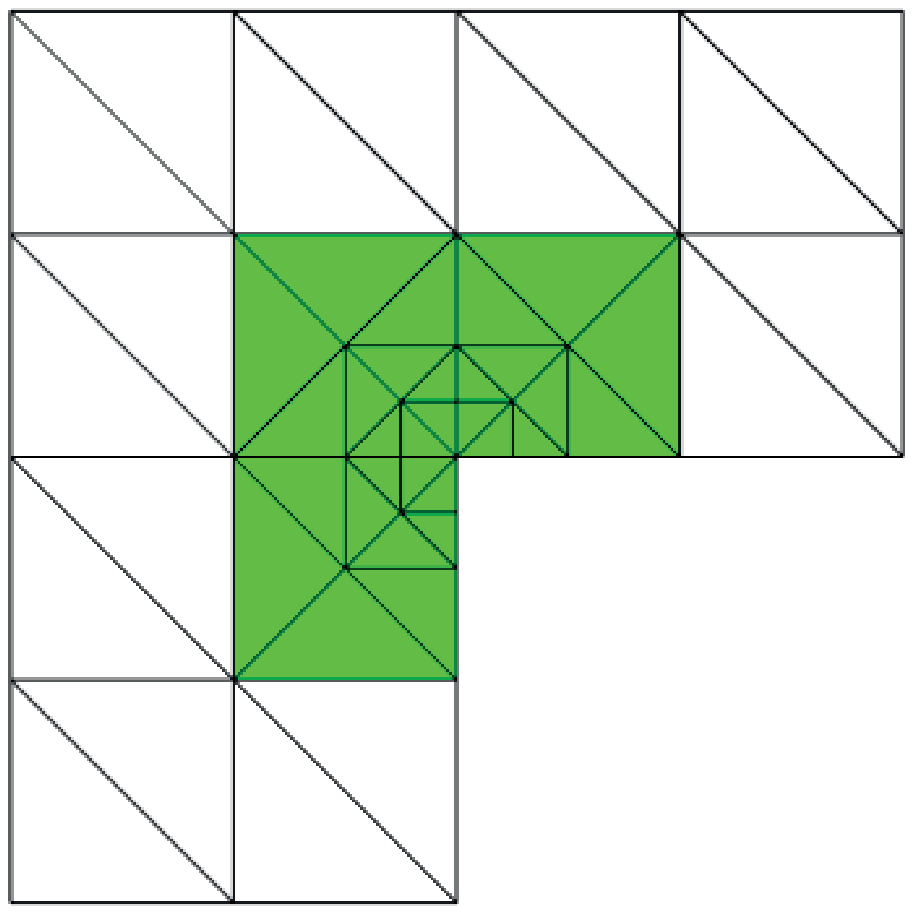}\\
(a) Initial mesh & (b) First refinement & (c) Second refinement
\end{tabular}
\caption{Initial coarse mesh
and local mesh refinement towards re-entrant corner. The fine region (in green) of the final mesh of form (c) always corresponds to the innermost 30 elements.}%
\label{fig:Mesh}%
\end{figure}

\begin{figure}[t]
\centering
\begin{tabular}
[c]{ccc}%
\includegraphics[width=.3\textwidth]{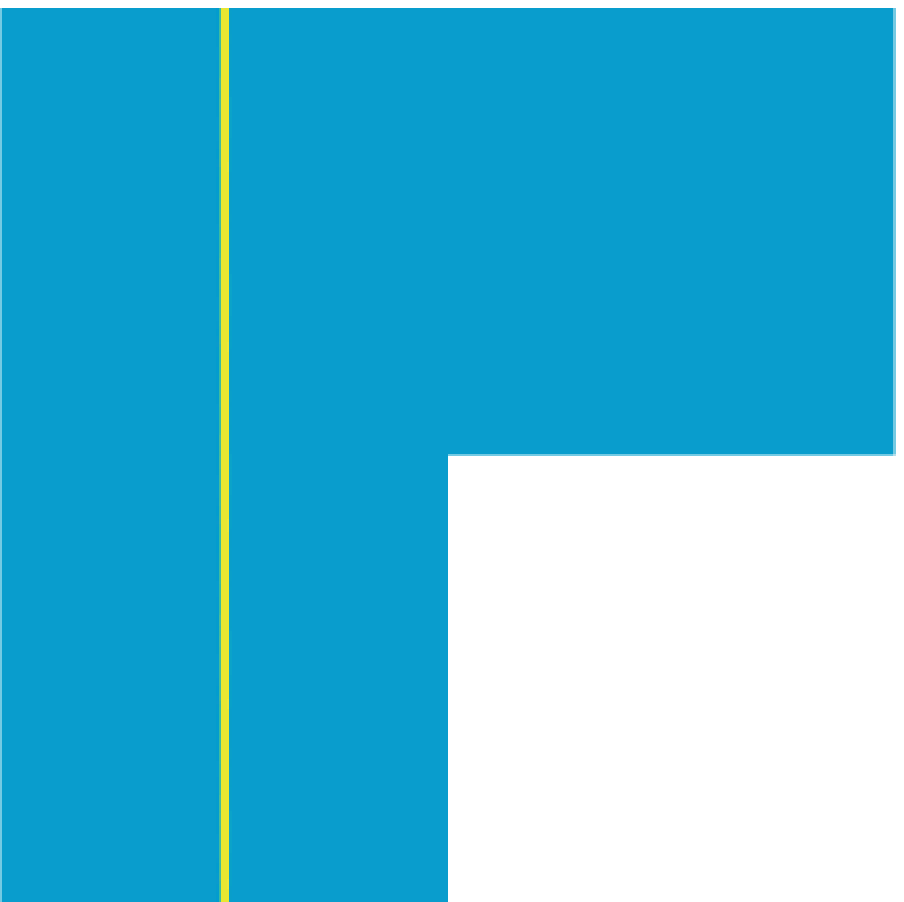} &
\includegraphics[width=.3\textwidth]{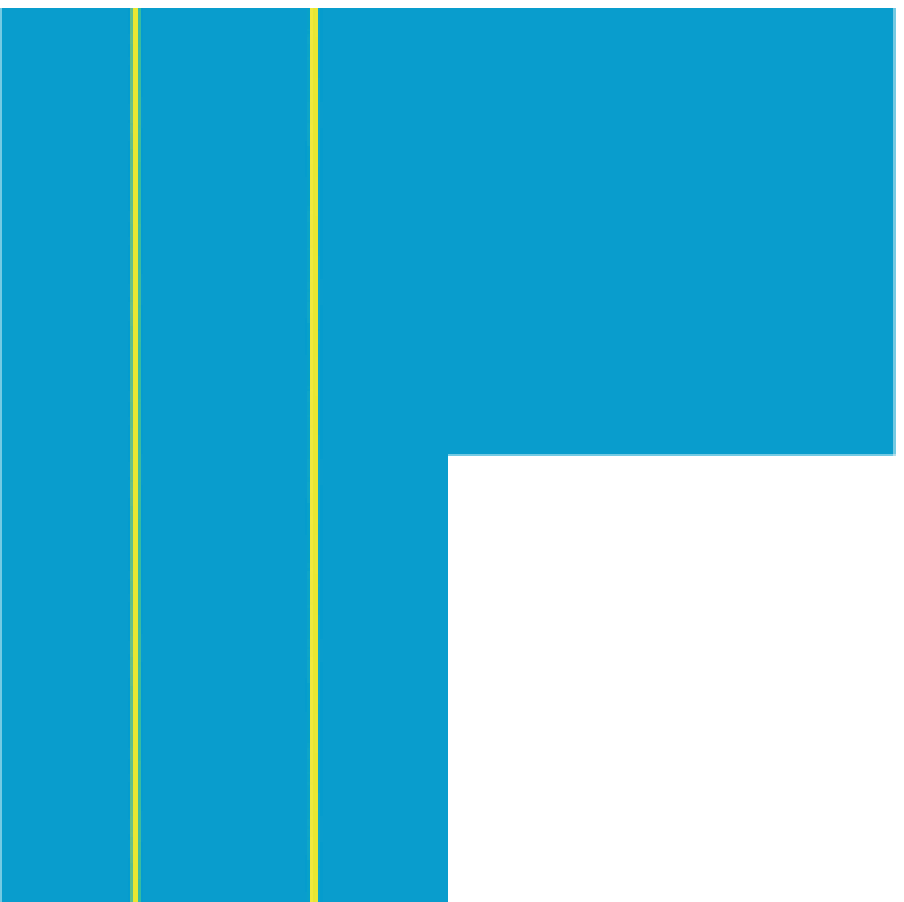} &
\includegraphics[width=.3\textwidth]{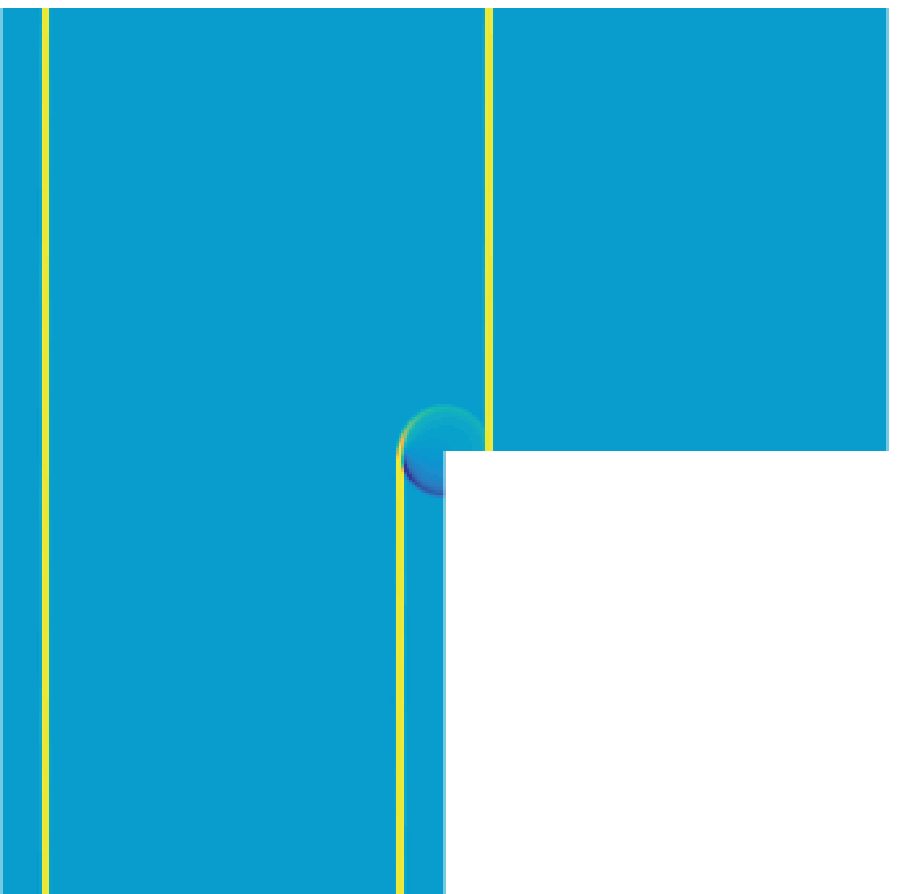}\\
\includegraphics[width=.3\textwidth]{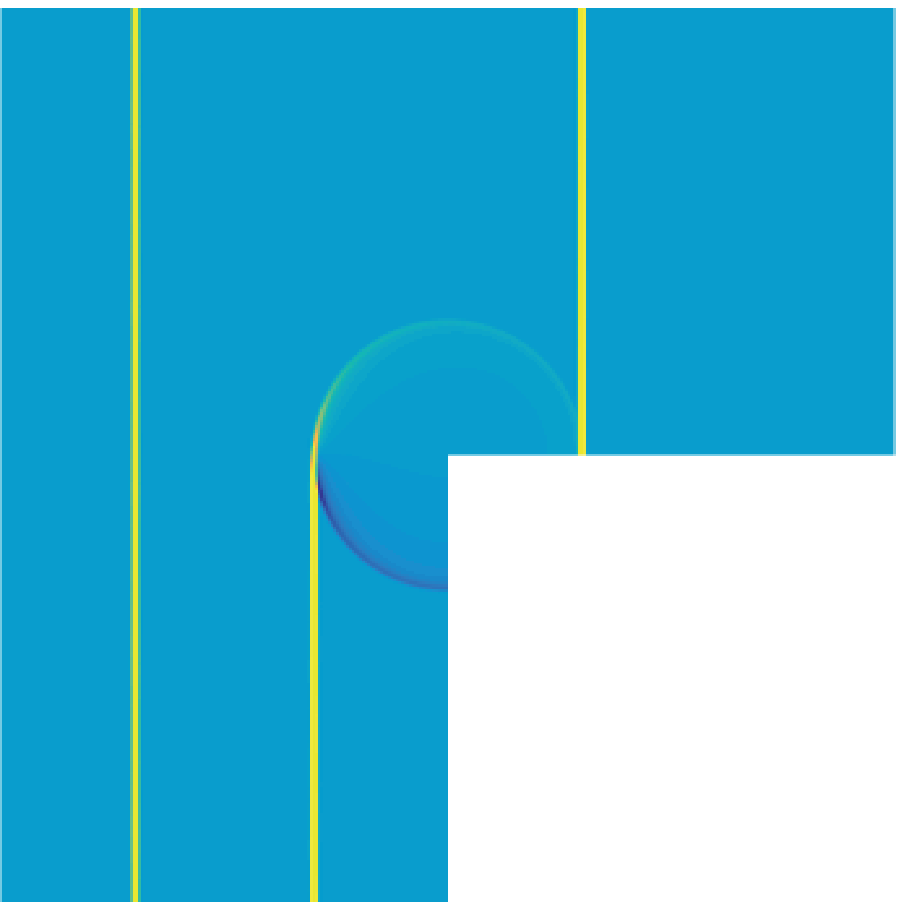} &
\includegraphics[width=.3\textwidth]{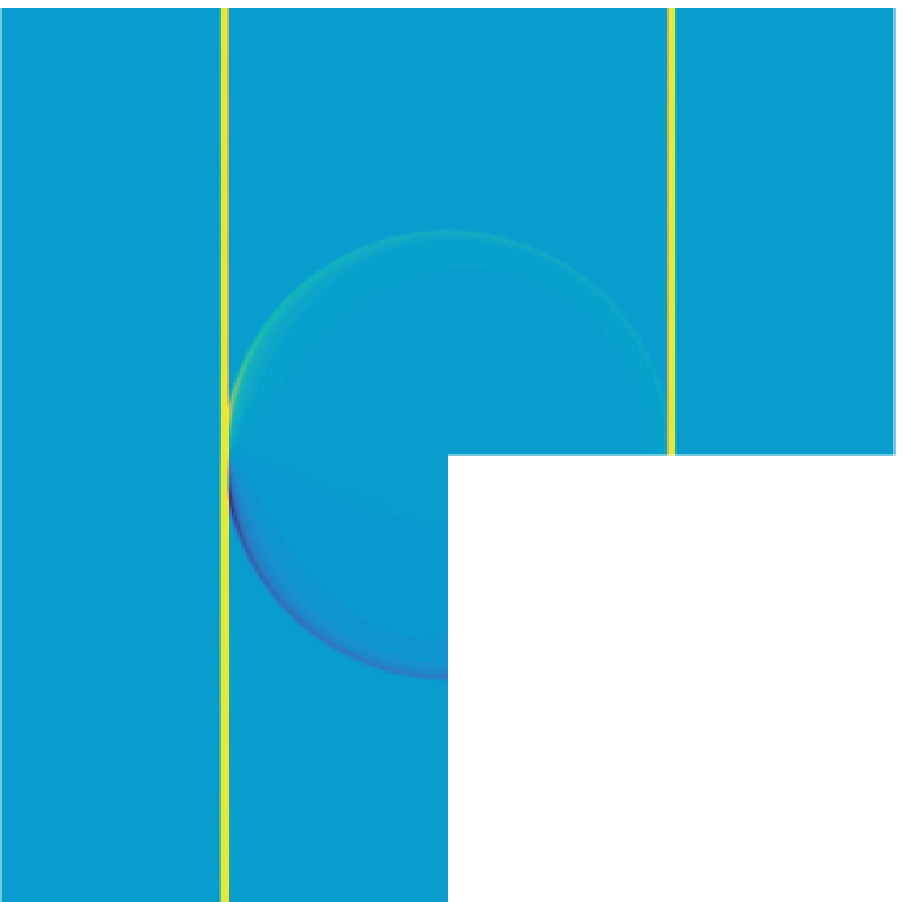} &
\includegraphics[width=.3\textwidth]{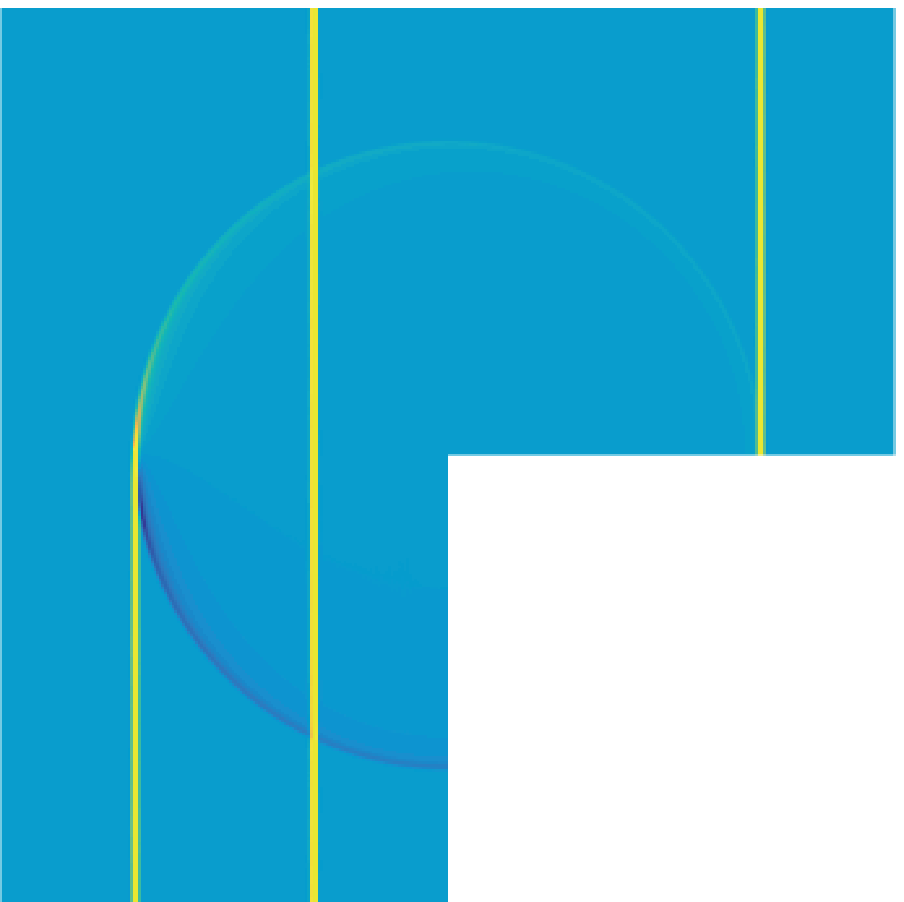}
\end{tabular}
\caption{Snapshots of the numerical solution at time $t = 0,\,
0.1,\,0.3,\,0.4,\,0.5,\,0.6$}%
\label{fig:Sol L shaped}%
\end{figure}
%

\section{Numerical Experiments\label{SecNumExp}}

Numerical experiments that corroborate the convergence rates and illustrate the stability properties of
the LTS-LF scheme when combined with continuous or discontinuous Galerkin FEM \cite{GSS06}
were presented in \cite{DiazGrote09}. Together with its higher order versions, the LTS-LF method
was also successfully applied to other (vector-valued) second-order wave equations from
electromagnetics \cite{Grote_Mitkova} and elasticity \cite{MZKM13,RPSUG15} .
Here we demonstrate the
versatility of the LTS approach in the presence of adaptive mesh refinement near a
re-entrant corner.

To illustrate the usefulness of the LTS approach, we consider the classical
scalar wave equation (Example \ref{Exmodel problem}) in the L-shaped domain
$\Omega$ shown in  Fig. \ref{fig:Mesh}. The re-entrant corner is located at
$(0.5,0.5)$ and we set $c=1$, $f=0$ and the final time $T=2$. Next, we
impose homogeneous Neumann boundary conditions on all
boundaries and choose as initial conditions the vertical Gaussian plane wave
\begin{equation*}
 u_0(x, y)  = \exp \left ( -(x-x_0)^2 / \delta^2 \right ),  \qquad v_0(x, y)  = 0, \qquad\qquad (x,y) \in \Omega\,,
 \end{equation*}
 of width $\delta=10^{-5}$ centered about $x_0 = 0.25$~.
 For the spatial discretization we opt for ${\mathcal P}^2$ continuous
finite elements with mass lumping \cite{CJRT01}.

First, we partition $\Omega$ into equal
triangles of size $h_{\mbox{\scriptsize init}}$ -- see Fig. \ref{fig:Mesh} (a).
Then we bisect the six elements nearest to the corner and subsequently bisect in the resulting mesh all elements with a vertex at $(0.5,0.5)$. Starting from that intermediate mesh, shown in Fig. \ref{fig:Mesh} (b),
we repeat this procedure again with the six elements adjacent to the corner, which
finally yields the mesh shown in Fig. \ref{fig:Mesh} (c). Hence the mesh refinement ratio,
that is the ratio between smallest elements in
the "coarse" and the "fine" regions, in the resulting mesh is 4:1. We therefore choose a four
times smaller time-step $\Delta\tau = \Delta t/p$ with $p = 4$
inside the fine region.

Clearly, this refinement strategy is heuristic, as optimal mesh refinement in the presence of corner singularities generally requires
hierarchical mesh refinement \cite{MS15}. However, when the region of local mesh refinement
itself contains a sub-region of even smaller elements, and so forth, any local time-step
will again be overly restricted due to even smaller elements inside the "fine" region.
To remedy the repeated bottleneck
caused by hierarchical mesh refinement, multi-level local time-stepping methods were proposed
in \cite{DG15,RPSUG15},
which permit the use of the appropriate time-step at every level of mesh refinement.
For simplicity, we restrict ourselves here to the standard (two-level) LTS-LF scheme.

 In Fig. \ref{fig:Sol L shaped} we display snapshots of the
numerical solution at different times: the plane wave splits into two wave fronts travelling in opposite
directions.
The lower half of the right propagating wave is reflected while the upper half proceeds into the
upper left quadrant.
To avoid any loss in the global CFL condition and reach the optimal global time-step,
we always include an overlap by one element, that is, we also advance the numerical solution inside
those elements immediately next to the "fine" region with the fine time-step.

In Fig. \ref{fig: Runtime} we compare the runtime of the LTS-LF($p$) on a
sequence of meshes using the refinement strategy depicted in Fig. \ref{fig:Mesh}, with the
runtime of a standard LF scheme with a time-step $\Delta t/4$ on the entire domain.
As expected, the LTS-LF method is faster than the standard LF scheme, in fact
increasingly so, as the number of refinements increases. Indeed, as the number of degrees
of freedom in the "coarse" region grows much faster than in the "fine" region, where it remains
essentially constant, the use
of local time-stepping becomes increasingly beneficial on finer meshes.

\begin{figure}[t]
\centering
\includegraphics[width = .7\textwidth]{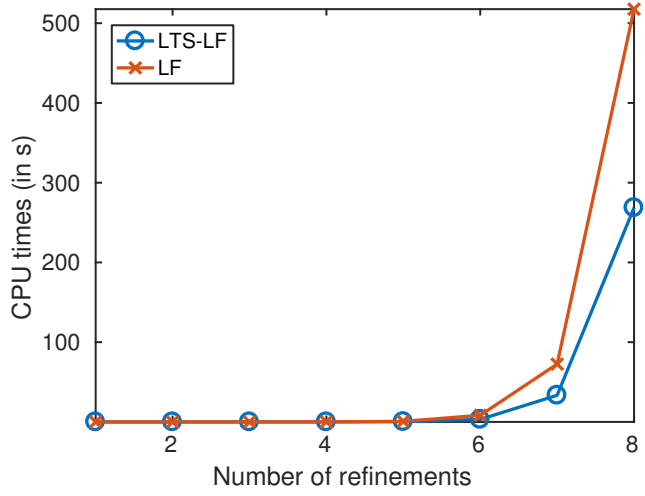}\caption{Comparison of run times between LTS-LF and standard LF vs. number of global refinements with constant coarse/fine mesh size ratio $p=4$.}%
\label{fig: Runtime}%
\end{figure}

{\bf Acknowledgements}

We thank Loredana Gaudio for useful comments and suggestions during the initial stages of this work and Maximillian Matth\"aus for his Matlab program.

\appendix

\section{Some Auxiliary Estimates}

\begin{lemma}
\label{LemTschebaschev} For $p\geq2$ let $\alpha_{j}^{p}$, $j=1,\dots,p-1$, be
recursively defined as in (\ref{rekalpha}). Then, the constants $\alpha
_{j}^{p}$ are given by
\begin{equation}
\alpha_{j}^{p}=\frac{%
{\displaystyle\prod\limits_{\ell=0}^{j}}
\left(  \ell^{2}-p^{2}\right)  }{(2j+2)!},\qquad1\leq j\leq p-1,\quad p\geq2
\label{resalpha}%
\end{equation}
Moreover, for $\kappa\in\left[  0,4p^{2}\right]  $ it holds%
\[
\left\vert \frac{2}{p^{2}}\sum_{j=1}^{p-1}\alpha_{j}^{p}\left(  \frac{\kappa
}{p^{2}}\right)  ^{j}\right\vert \leq\frac{\kappa}{12}\quad\text{and\quad
} \left\vert \frac{2}{p^{2}}\sum_{j=1}^{p-1}\alpha_{j}^{p}\left(  \frac{\kappa}{p^{2}%
}\right)  ^{j-1} \right\vert \leq\frac{p^{2}-1}{12}.
\]

\end{lemma}%

\proof
To show that the constants $\alpha_{j}^{p}$ are in fact given by
\eqref{resalpha}, we first use the identity
\begin{equation}
p (p+j)(p+j-1) \dots(p+1)p(p-1)\dots(p-j+1)(p-j) =
{\displaystyle\prod\limits_{\ell=0}^{j}}
\left(  p^{2}-\ell^{2}\right)
\end{equation}
to rewrite \eqref{resalpha} as
\begin{equation}
\alpha_{j}^{p}=\frac{\left(  -1\right)  ^{j+1}p\left(  p+j\right)  !}{\left(
p-j-1\right)  !\left(  2j+2\right)  !}. \label{shortdef}%
\end{equation}
By using \eqref{shortdef} it is then straightforward to verify that
$\alpha_{j}^{p}$ satisfies the recursive definition in (\ref{rekalpha}).

Next, one proves by induction that
\begin{align*}
\sum_{j=1}^{p-1}\alpha_{j}^{p}x^{j}  &  =\frac{p^{2}}{2}+\frac{T_{p}\left(
1-\frac{x}{2}\right)  -1}{x}\\
\sum_{j=1}^{p-1}\alpha_{j}^{p}x^{j-1}  &  =\frac{p^{2}x+2\,T_{p}\left(
1-\frac{x}{2}\right)  -2}{2x^{2}}.
\end{align*}
with the \v{C}eby\v{s}ev polynomials $T_{p}$ of the first kind. We recall that
\begin{equation}
T_{p}^{\left(  m\right)  }\left(  1\right)  =%
{\displaystyle\prod\limits_{\ell=0}^{m-1}}
\frac{\left(  p^{2}-\ell^{2}\right)  }{\left(  2\ell+1\right)  }%
\qquad\text{and\quad}\left\Vert T_{p}^{(m)}\right\Vert _{L^{\infty}\left(
\left[  -1,1\right]  \right)  }=T_{p}^{\left(  m\right)  }\left(  1\right)  ,
\label{Tpproperties}
\end{equation}
where the first relation follows from \cite[(1.97)]{Rivlin} and the second one
from \cite[Theorem 2.24]{Rivlin}, see also \cite[Corollary 7.3.1]%
{SauterSchwab2010}.

Now, let $x=\kappa/p^{2}$. The condition $\kappa\in\left[
0,4p^{2}\right]  $ implies $\left[  1-\frac{x}{2},1\right]  \subset\left[
-1,1\right]  $. Hence, a Taylor argument shows that there exists $\xi
\in\left[  -1,1\right]  $ such that%
\begin{align}
\left\vert \sum_{j=1}^{p-1}\alpha_{j}^{p}x^{j} \right\vert &  =\left\vert \frac{p^{2}}{2}+\frac{T_{p}\left(
1\right)  -\frac{x}{2}T_{p}^{\prime}\left(  1\right)  +\frac{x^{2}}{8}%
T_{p}^{\prime\prime}\left(  \xi\right)  -1}{x}\right\vert \nonumber\\
&  =\left\vert \frac{x}{8}T_{p}^{\prime\prime}\left(  \xi\right) \right\vert  \leq\frac{p^{2}\left(
p^{2}-1\right)  }{24}x=\frac{p^{2}-1}{24}\kappa,\label{estsum}%
\end{align}
where we have also used \eqref{Tpproperties}.
Similarly, we get%
\begin{align*}
\left\vert \sum_{j=1}^{p-1}\alpha_{j}^{p}x^{j-1} \right\vert &  =\left\vert \frac{p^{2}x+2\left(  T_{p}\left(
1\right)  -\frac{x}{2}T_{p}^{\prime}\left(  1\right)  +\frac{x^{2}}{8}%
T_{p}^{\prime\prime}\left(  \xi\right)  \right)  -2}{2x^{2}}\right\vert \\
&  =\left\vert \frac{p^{2}x+2\left(  1-\frac{xp^{2}}{2}+\frac{x^{2}}{8}T_{p}%
^{\prime\prime}\left(  \xi\right)  \right)  -2}{2x^{2}}\right\vert =\frac{1}{8}%
\left\vert T_{p}^{\prime\prime}\left(  \xi\right)\right\vert   \leq\frac{p^{2}\left(  p^{2}-1\right)
}{24}.
\end{align*}%
\endproof

\bibliographystyle{abbrv}
\bibliography{nlailu,refs,references}

\end{document}